\documentclass[11pt,reqno,a4paper]{amsart}

\usepackage{xcolor}

\ifdefined\DARKMODE

\pagecolor[rgb]{0,0,0} 
\color[rgb]{0.5,0.5,0.5} 
\fi

\usepackage{amsmath,amsfonts,amsthm,amssymb}

\usepackage[
colorlinks,
pdfpagelabels,
pdfstartview = FitH,
bookmarksopen = true,
bookmarksnumbered = true,
linkcolor = blue,
plainpages = false,
hypertexnames = false,
citecolor = red] {hyperref}

\usepackage[capitalize]{cleveref}
\usepackage[T1]{fontenc}

\usepackage{graphicx}
\usepackage{cite}

\usepackage{mathabx}
\usepackage{esint}

\usepackage{subfigure}

\makeatletter
\def\section{\@startsection{section}{1}%
\z@{.7\linespacing\@plus\linespacing}{.5\linespacing}%
{\bfseries
\centering
}}
\def\@secnumfont{\bfseries}
\makeatother

\usepackage{graphicx}

\usepackage[capitalize]{cleveref}
\crefname{enumi}{}{}
\crefname{equation}{}{}

\usepackage[final]{showlabels}

\usepackage{comment}

\includecomment{comment}

\specialcomment{supplementary}
{\colorlet{savedcolor}{.} \color{blue} \begingroup \ttfamily      \noindent \underline{Supplementary details:} \newline \newline \footnotesize }{\endgroup   \color{savedcolor}}

\newcommand{\EX}{\mathbb{E}}
\newcommand{\N}{\mathbb{N}}
\newcommand{\Z}{\mathbb{Z}}

\newcommand{\E}{\mathbb{E}}

\newtheorem{theorem}{Theorem}[section]
\newtheorem{lemma}[theorem]{Lemma}
\newtheorem{proposition}[theorem]{Proposition}
\newtheorem{corollary}[theorem]{Corollary}

\theoremstyle{definition}
\newtheorem{definition}[theorem]{Definition}
\theoremstyle{remark}
\newtheorem{remark}[theorem]{Remark}
\newtheorem{example}[theorem]{Example}
\numberwithin{equation}{section}
\usepackage{enumerate}

\newcommand{\eps}{\varepsilon}

\usepackage{bbm}

\newcommand{\R}{\mathbb{R}}

\newcommand{\T}{\mathbb{T}}

\newcommand{\grad}{\nabla}
\newcommand{\norm}[1]{\ensuremath{\Vert #1 \Vert}} 
\newcommand{\abs}[1]{\ensuremath{\vert #1 \vert}}  

\DeclareMathOperator{\dist}{dist}

\DeclareMathOperator*{\osc}{osc}

\renewcommand{\a}{\mathbf{a}}

\begin{document}

\title[Renormalized stochastic pressure equation]{Renormalized stochastic pressure equation with log-correlated Gaussian coefficients}

\author[Avelin]{Benny Avelin}
\address{Uppsala University, Department of Mathematics, 751 06 Uppsala, Sweden}
\email{benny.avelin@math.uu.se}

\author[Kuusi]{Tuomo Kuusi}
\address{University of Helsinki, Department of Mathematics and Statistics, PO Box 68 (Pietari Kalmin katu 5), 00014 University of Helsinki, Finland}
\email{tuomo.kuusi@helsinki.fi}

\author[Nummi]{Patrik Nummi}
\address{Aalto University School of Business,
Department of Information and Service Management, P.O. Box 21210 (Ekonominaukio 1, Espoo), 00076 Aalto, Finland}
\email{patrik.nummi@aalto.fi}

\author[Saksman]{Eero Saksman}
\address{University of Helsinki, Department of Mathematics and Statistics, PO Box 68 (Pietari Kalmin katu 5), 00014 University of Helsinki, Finland}

\email{eero.saksman@helsinki.fi}

\author[T\"olle]{Jonas M. T\"olle}
\address{Aalto University School of Science, Department of Mathematics and Systems Analysis, PO Box 11100, 00076 Aalto, Finland}
\email{jonas.tolle@aalto.fi}

\author[Viitasaari]{Lauri Viitasaari}
\address{Aalto University School of Business,
Department of Information and Service Management, P.O. Box 21210 (Ekonominaukio 1, Espoo), 00076 Aalto, Finland}
\email{lauri.viitasaari@aalto.fi}

\thanks{BA acknowledge support by the Swedish Research Council [dnr: 2019-04098]. TK and PN acknowledge support by the Academy of Finland and the European Research Council (ERC) under the European Union's Horizon 2020 research and innovation programme (grant agreements no. 323099 and no. 818437). JMT acknowledges support by the Academy of Finland and the European Research Council (ERC) under the European Union's Horizon 2020 research and innovation programme (grant agreements no. 741487 and no. 818437). TK and ES acknowledge support by the Academy of Finland CoE FiRST (no. 1346305)}
\date{\today}
\begin{abstract}
	We study periodic solutions to the following divergence-form stochastic partial differential equation with Wick-renormalized gradient on the $d$-dimensional flat torus $\T^d$,
	\[
		-\nabla\cdot\left(e^{\diamond (-  \beta X) }\diamond\nabla U\right)=\nabla \cdot (e^{\diamond (-  \beta X)} \diamond \mathbf{F}),
	\]
 where $X$ is the log-correlated Gaussian field,  $\mathbf{F}$ is a random vector \textcolor{black}{field representing the flux, the in/out-flow of fluid per unit area per unit time,} and $\diamond$ denotes the Wick product.
 	The problem is a variant of the stochastic pressure equation, in which $U$ is modeling the pressure of a creeping water-flow in crustal rock that occurs in enhanced geothermal heating. 
	In the original model, the Wick exponential term
	$e^{\diamond(-\beta X)}$ is modeling the random permeability of the rock.
	The porosity field is given by a log-correlated Gaussian random field
	$\beta X$, where $\beta<\sqrt{d}$. We use elliptic regularity theory in order to define a notion of a solution to this (a priori very ill-posed) problem, via modifying the $S$-transform from Gaussian white noise analysis, and then establish the existence and uniqueness of solutions. Moreover, we show that the solution to the problem can be expressed in terms of the Gaussian multiplicative chaos measure.
\end{abstract}


\maketitle

{\small
\tableofcontents
}

\section{Introduction}
The stochastic pressure equation on a domain $V\subset\R^d$ is given by the following stochastic partial differential equation (SPDE)
\begin{equation}\label{eq:general}\begin{cases}&-\nabla\cdot\left(e^{\xi} \times \nabla U\right)= \nabla \cdot (e^{\xi} \times \textcolor{black}{\mathbf{F}}),\\
&\text{boundary conditions on $\partial V$},\end{cases}\end{equation}
where $\xi$ is a realization of a Gaussian random field on $V$. The product ``$\times$'' is usually given by some algebraic renormalization, where many authors have considered the Wick-product ``$\diamond$'', and set $e^\xi=e^{\diamond W}$, the Wick exponential of white noise $W$, see~\cite{Holden_et_al} and the references therein.

We are studying the question of existence and uniqueness of solutions to the following divergence-form SPDE

\begin{equation}\label{eq:mainmain}
    -\nabla\cdot\left(e^{\diamond (-  \beta X) }\diamond\nabla U\right)=\nabla\cdot\left(e^{\diamond (-  \beta X) }\diamond \mathbf{F}\right),
\end{equation}
on the $d$-dimensional flat torus $\mathbb{T}^d$ with dimension $d\geq 2$. 
Here $0<\beta<\sqrt{d}$ is a parameter, and $X$ is realization of a log-correlated Gaussian random field on $\T^d$ \textcolor{black}{arising from a fractional Gaussian free field, see \eqref{eq:field-definition} for precise definition}, and $\mathbf{F}$ is a suitable random vector \textcolor{black}{field that is also generated by the underlying Gaussian source of randomness}. This continues our work \cite{AKNSTV-1} in which we study the one-dimensional case. 

The log-correlated Gaussian field\textcolor{black}{s} have a covariance function with a specific logarithmic-type singularity on the diagonal:
\begin{equation}
	\E \left[X(y)X(z)\right] = \log \frac{1}{|y-z|_{\T^d}} + g(y,z),
\end{equation}
where $g$ is typically assumed to be smooth, and $|\cdot|_{\T^d}$ denotes the distance in the torus. It is clear that $X$ cannot be a function (but is close to being one, see \cite{junnila2020}), and in fact \textcolor{black}{such} $X$ can be obtained as a special case of the so-called fractional Gaussian fields\textcolor{black}{, formally given by $X = (-\Delta)^{\alpha}W$, where $\Delta$ is the Laplacian, and $W$ denotes the (real) white noise. Then, one recovers the log-correlated Gaussian field with $\alpha = -d/4$}, see for instance \cite{SheffieldSurvey}. The inclusion of the log-correlated field in the problem is motivated by geophysical modeling, which we discuss in detail in Subsection~\ref{subsec:geo} below.


In general, we do not expect weak solutions to~\eqref{eq:mainmain} to have sufficient regularity to compensate for the lack of regularity of the random factor, namely the log-correlated Gaussian field. The known fact that there is no canonical way to define the product of two distributions, is tackled by ``renormalizing'' the equation. This process can be viewed a local reparametrization leading to an implicit dependence of the solution to the specific choice of smooth approximations. In recent years, this method has been applied systematically in the context of so-called singular SPDEs in~\cite{Bruned_et_al, BHZ, Hairer, Kupiainen, Gubinelli_et_al, OttoWeber}. 

In \eqref{eq:mainmain} we work, albeit indirectly, with the Wick product, which most resembles the Itô stochastic integral; see \cite{HuOe1996} for details. As already observed by \cite[pp. 141--191]{Holden_et_al} in the white noise case, the choice of the Wick product ensures that the stochastic pressure equation remains unbiased, given deterministic data. This property aligns with the local martingale property of the Itô integral in the one-dimensional case. The Wick product is well-known in particle physics from Wick's theorem \cite{Wick}, which can be reformulated via Gaussian integration by parts \cite{Janson1997}. This reformation aids in deriving the conformal Ward identities in Liouville quantum field theory \cite{KRV, Rhodes:2016hin}. An earlier appearance of the Wick product can be found in Isserlis' theorem concerning higher-order moments of the multivariate Gaussian normal distribution \cite{I:18}.

\textcolor{black}{While intimately connected, Wick's theorem and the general Wick product consider different concepts. Renormalization through the Wick product naturally leads to the so-called $S$-transform, which mimics a variant of the Laplace transform with random test-elements. Indeed, $S$-transform on Wick products satisfy
$$
\left(S\left(X \diamond Y\right)\right)(h) = (SX)(h)(SY)(h)
$$
for test-elements $h$, and thus Wick products behave as ordinary products under $S$-transform. Moreover, $S$-transforms characterize random variables uniquely, see, e.g., \cite{Janson1997} in the case of Gaussian $L^2(\Omega)$-setting.
For an introduction to the connections between the Wick product, Wiener chaos, and the $S$-transform in the classical context of Gaussian linear spaces, see also the survey \cite{Hu2009Wick}. These tools of stochastic renormalization can be extended beyond the $L^p$-setting into so-called Hida and Kondratiev stochastic distribution spaces; we refer to \cite{kondratiev1996,hida2013white} and the references therein for more details.}

\textcolor{black}{In the present article,} we define (a variant of) the $S$-transform at test points $X_\beta(z)$ in the $L^2_\beta(\Omega)$ space of random variables of form 
\begin{align*}
    \int_{\T^d} \varphi(z) d\mu_\beta(z), \quad \varphi \in C^\infty(\T^d),
\end{align*}
where $d\mu_\beta(z)$ denotes the Gaussian multiplicative chaos measure (\textcolor{black}{For details on the Gaussian multiplicative chaos measure, see \cite{Rhodes-Vargas-2014, rhodes2016} and Section \ref{subsec:GMC-preliminaries} below}) with parameter $\beta$. Applying the $S$-transform to equation~\eqref{eq:mainmain} results in a family of weighted, divergence-form, deterministic PDEs, parameterized by $z\in \T^d$:
\begin{equation}
\label{eq:deterministic-eq-intro}
		\begin{cases}
			- \nabla \cdot (w(\cdot; z) \nabla u(\cdot; z)) = \nabla \cdot (w(\cdot; z) \mathbf{f}(\cdot;z)), & \text{ in $\T^d$,} \\
			y \mapsto u(y;z) \textnormal{ is } \T^d \textnormal{-periodic for all }z \in \T^d
		\end{cases}
\end{equation}
for the weight $w(y;z):=|y-z|^{\beta^2}e^{-\beta^2 g(y,z)}$, which can be tackled by well-known methods of degenerate elliptic PDEs. \textcolor{black}{Here $u(\cdot,z)$ corresponds to the $S$-transform at $X_\beta(z)$ of the stochastic solution $U(y)$ to \cref{eq:mainmain}, and $\mathbf{f}(\cdot;z)$ corresponds to the $S$-transform at $X_\beta(z)$ of $\mathbf{F}$ in \cref{eq:mainmain}.} We show that these deterministic equations admit unique solutions $u(\cdot;z)$, which have suitable regularity in the $z$ variable, and this in turn allows the inversion of the $S$-transform, provided that the $S$-transform of $\mathbf{F}$ has sufficient regularity. \textcolor{black}{The regularity of the solution to \eqref{eq:deterministic-eq-intro} in the $z$ variable is studied in \cref{sec:regularity}. We emphasize that in our context, we study the regularity in the parametrization variable $z$ which is in contrast with classical regularity theory for degenerate PDEs. Our strategy to obtain sufficient regularity is to shift the pole whenever $x$ is far from a given $z$. We construct a corrector expansion with respect to this shifting that allows us to derive sufficient estimates in the $L^\infty$ norm. These then imply sufficient regularity in $z$, see \cref{sec:regularity} for details.} This leads to our main result on the existence, uniqueness and representation of \textcolor{black}{stochastic} solutions to \eqref{eq:mainmain}, which is outlined below:



\begin{theorem}\label{thm:main}
	Let $X_\beta$ be a log-correlated Gaussian field with scaling parameter $0 < \beta < \sqrt{d}$, and let $\mathbf{F}$ be a random vector with sufficient regularity on its $S$-transform.  Then the problem~\eqref{eq:mainmain} admits a solution (in the sense of Definition \ref{def:pde-solution}). In particular, the solution admits a representation given by 
 \begin{align}\label{eq:mainthm}
		U(y) = \int_{\mathbb{T}^d} \varphi(y;z) d\mu_\beta(z), \qquad y \in \mathbb{T}^d,
	\end{align}
 where $\varphi$ is deterministic, and where $d\mu_\beta$ denotes the Gaussian multiplicative chaos measure with parameter $\beta$.

\end{theorem}
For the precise statement and the proof of the main result, see Theorem \ref{thm:main-result} in Section \ref{sec:mainresult} below.

\textcolor{black}{The above result states that the problem~\eqref{eq:mainmain} with Wick product admits a stochastic solution that can be viewed, for each fixed $y$, as a well-defined square integrable random variable $U(y)$. Our key observation is that if the probabilistic setup is appropriately understood, one can understand the solution through a deterministic PDE~\eqref{eq:deterministic-eq-intro}. Then our main contribution is to show sufficient regularity for the solution to~\eqref{eq:deterministic-eq-intro}, and one does not need many probabilistic tools in order to obtain the solution. We note however that, as pointed out above, Wick products and $S$-transforms can be defined more generally in stochastic distribution spaces. In this sense, one could approach the problem~\eqref{eq:mainmain} from a more probabilistic point of view, as is done in the univariate case \cite{AKNSTV-1}. 
Another natural possibility is to regularize the field $X$, e.g., via mollification and consider \cref{eq:mainmain} with the regularized field $X^{\varepsilon}$. This approach is also considered in the univariate case \cite{AKNSTV-1} in which one can compute the solution explicitly. However, in multivariate case one cannot write down the solution with explicit formula, making it difficult to assess what would happen in the limit. This is an interesting direction of future research. Other possible future directions may include the analysis of rough paths inspired expansions \cite{FH:20}, more elaborate methods as regularity structures \cite{Hairer}, or more direct approaches by limits of regularizations. Taking this viewpoint, the Wick product may be perceived as the first order term in a more general formal expansion. We point out, however, that by an easy power counting argument, the parabolic counterpart to \eqref{eq:general} is seen to be supercritical for $d\ge 2$.
The renormalization issue thus remains delicate. From the physical point of view, as discussed in the following subsection, the Dirichlet problem on a bounded domain certainly could be an interesting topic of investigation.
}


\subsection{Relation to the physical model}\label{subsec:geo}

Motivated by the St1 Deep Heat project in Espoo~\cite{Kukkonen_2021}, we are interested in modeling the high pressure creeping water-flow in crustal rock that occurs in enhanced geothermal heating (EGH). The EGH system consists of two 7-km-deep boreholes into the bedrock of the Nordic countries. Water is pushed into one hole and creeps through the bedrock into the second hole, heated by the geothermal heat in the process, such that the hot water can be extracted for energy. Understanding the fluid flow is essential for understanding the heat extraction.

The modeling of fluid flow in rock has been widely studied, mainly due to the oil industry, for instance in the stimulation of oil wells. For porous rock types the homogenized problem describes the flow fairly well~\cite{L:90,FW:92,LW:94}.
It has long been known that the creeping water-flow in porous media can be well described by Darcy's law, given as
\begin{align*}
	\frac{\kappa(x)}{\nu} \grad P  = v(x).
\end{align*}
Here~$v(x)$ describes the flow velocity, $P$ is the pressure, $\kappa$ is the permeability and $\nu$ the (constant) viscosity of water. From preservation of ``mass/energy'' we get that the pressure $P$ of the creeping flow satisfies the following parabolic equation
\begin{align*}
	\frac{\partial P}{\partial t}-\nabla \cdot (\nu^{-1} \kappa(x)\grad P(x,t) ) = f\textcolor{black}{,}
\end{align*}
where \textcolor{black}{$f$ is the source rate of the fluid.} As such, in our case we are interested in the steady state this system, i.e.~where $\partial P / \partial t = 0$, specifically
\begin{align} \label{eq:pressure-intro}
	-\nabla \cdot (\nu^{-1} \kappa(x)\grad P(x,\cdot) ) = f.
\end{align}

In engineering applications, we cannot measure the porosity/permeability of the rock other than locally around the borehole. Hence the uncertainty about the surrounding rock is modeled as a random field. It has been observed from the borehole measurements that the spatial correlation of the crustal crack distribution follows a power-law scaling that is called $1/k$ scaling. This was actually first noted by Hewett et al.~\cite{H} for the variances of borehole logs of oil and gas reservoir formations. Todoeschuk et at.~\cite{T} further noted that the sequence of seismic reflection coefficients computed from a well log had a power law spatial frequency spectrum. Leary~\cite{L90,L91}, on the other hand, found the power law scaling spectra for deep borehole spatial fluctuation sequences in sonic velocity and resistivity values closely related to fracture content in crustal granite. Recently, the power-law scaling has also been observed in wellbore porosity logs, see Leary et al.~\cite{leary2020physical} and the references therein. While the borehole measurements only contains information in one direction (depth), we can view them as a 1d projection of the 3d field modeling the rock porosity.

The permeability is observed to be log-normally distributed and is related to the  exponential of porosity (see~\cite{leary2020physical} and the references therein). Thus, a natural model for permeability would be $\kappa = e^{\alpha \gamma + \beta \phi}$ for some given parameters $\alpha,\beta,\gamma > 0$, where $\gamma$ is mean porosity and $\phi$ is normally distributed. In order to have the correct $1/k$ scaling properties suggested by empirical evidence, a natural choice (in dimension $d=3$) is to consider $\phi \sim \text{FGF}_{3/2}(\mathbb{R}^3)$ with $\text{FGF}$ standing for fractional Gaussian field that corresponds to the log-correlated Gaussian field. However, this is not really well defined (we need to consider a suitably renormalized exponential). Furthermore, $\alpha \gamma + \beta\phi$ cannot, strictly speaking, be interpreted as porosity which can only take values between 0 and 1. However, since the permeability can be arbitrarily close to zero (solid rock), this seems to be a plausible model for permeability. In our approach we consider arbitrary dimension $d\geq 2$ (the case $d=1$ already being treated in \cite{AKNSTV-1}), and we need a bound $\beta < \sqrt{d}$.
Notice carefully, that we assume above that $\phi$ has zero mean, so in practice the choice of $\beta$ scales the size of fluctuations of porosity, not the absolute size. For example, in~\cite{leary2020physical}, $\alpha \gamma$ is observed to be of size $3$ to $4$, and in our analysis this is simply a multiplicative constant for the equation and it plays no significant role. 

On a related literature, let us comment that a discussion about mathematical modeling of the permeability with the exponential of a Gaussian field would not be complete without mentioning the works of Holden et al., see~\cite{Holden_et_al} and the references therein. They use the exponential of smoothed white noise as the permeability. In comparison to our approach, the unfortunate side-effect of using smoothed white noise is that there is a finite correlation length which violates the correlation length of cm to km scale from~\cite{leary2020physical}. 

\subsection{Notation}
\label{subsec:notation}
Throughout the article we use the following notation. Let $f\in L^1(\T^d)$. Then the Fourier coefficients of $f$ are given by, for $ n\in \Z^d$, 
$$
\widehat{f}(n) = \int_{\T^d} f(x)e^{-ix\cdot n}dx.
$$
\begin{definition}
\label{def:frac-sobolev-general}
    Let $s \in \R$, then $W^{s,2}(\T^d)$ is defined as
    \begin{equation*}
        W^{s,2}(\T^d) = \{f \in S'(\T^d) \textnormal{ s.t. } \Vert f\Vert_{W^{s,2}(\T^d)}^2 =  \sum_{n\in \Z^d} \langle n \rangle^{2s}|\hat{f}(n)|^2 < \infty\},
    \end{equation*}
    where $S'(\T^d)$ is the space of tempered distributions on $\T^d$, and $\hat{f}(n)$ denotes the Fourier coefficients of $f$, and $\langle n\rangle^2 = 1+|n|^2 = 1+\sum_{k=1}^d n_k^2$.
\end{definition}
We also need to define a suitable distance on the torus $\T^d$. For this purpose, we define $\varphi: \R^d \mapsto \R$ as a $2\pi$-periodic function in each variable such that $\varphi(u)=|u|$ when $|u|<1/3$ and $\varphi(u)$ is positive and $C^\infty$-smooth in the set $\R^d\setminus (2\pi\Z)^d$. Then $\varphi$ yields function $\widetilde{\varphi}$ on $\T^d$ satisfying 
$$
\widetilde{\varphi}(x+(2\pi\Z)^d) = \varphi(x),\quad x\in \R^d.
$$
Then we may define the \textcolor{black}{''distance''} between two points $x,y\in\T^d$
by\begin{equation}
\label{eq:torus-distance-decomposition}
d_{\T^d}(x,y) := |x-y|_{\T^d} := \widetilde{\varphi}(x-y).
\end{equation}
\textcolor{black}{We note that this does not automatically yield a proper distance for any such $\varphi$. However, it yields a distance on small scales which is sufficient for our purposes, as on small scales it behaves exactly as Euclidean distance. Moreover, it is smooth outside the diagonal. }

In the sequel, we will simply write $\Vert f \Vert_{s,2}$ instead of $\Vert f\Vert_{W^{s,2}(\T^d)}$ whenever confusion cannot arise. We also note that the topological dual of $W^{s,2}(\T^d)$ (in the standard distributional duality sense) is simply $W^{-s,2}(\T^d)$. 

For a weight $w$ and a domain $\Omega \subset \T^d$, we also use the notation $W^{m,p}(\Omega;w)$, where $1\leq p< \infty$, for the weighted Sobolev space defined as the completion of  $C^\infty(\Omega)$ under the norm 
\begin{align}
\label{eq:weighted-Sobolev}
	\|u\|_{W^{m,p}(\Omega;w)}
	:= \|u\|_{L^p(\Omega;w)} + \sum_{1 \leq |\alpha| \leq m} \|\partial_x^\alpha u\|_{L^p(\Omega;w)},
\end{align}
where $\alpha$ is a multi-index and where
\begin{align*}
	\|u\|_{L^p(\Omega;w)} := \left ( \int_{\Omega} |u(x)|^p w(x) dx\right )^{\frac{1}{p}}.
\end{align*}
We will often use the shorthand notation $d\textcolor{black}{\nu}_z = w(x) dx$, and we will, as customary, denote $H^1(\Omega;w) := W^{1,2}(\Omega;w)$.



\subsection{The plan}
The rest of the article is organized as follows. In Section \ref{sec:log-correlated} we discuss basic facts on log-correlated fields and Gaussian multiplicative measures, cf. Section \ref{subsec:GMC-preliminaries}, and we define the $S$-transform on log-correlated fields and study its basic properties, cf. Section \ref{subsec:S-transform}. In Section \ref{sec:mainresult} we motivate and describe precisely what we mean by a solution to the problem \eqref{eq:mainmain}. We also formulate and prove our main result, by using regularity estimates for the deterministic equation \eqref{eq:deterministic-eq-intro}. Section \ref{sec:regularity} is devoted to the required regularity estimates, and we end the paper with a short appendix on basic facts about Muckenhoupt weights needed in Section \ref{sec:regularity}.
\section{Log-correlated fields and the $S$-transform}
\label{sec:log-correlated}
\label{sec:probability}
\subsection{Basic facts on log-correlated fields and Gaussian multiplicative chaos measures}
\label{subsec:GMC-preliminaries}
In this section we recall some basic facts on log-correlated Gaussian fields and Gaussian multiplicative chaos (GMC) measures in the $d$-torus. For details on GMC measures in general, we refer to survey articles~\cite{Duplantier2017},~\cite{Rhodes-Vargas-2014} and references therein. For log-correlated fields in the $d$-torus, see also~\cite{Oh2020}.

\begin{definition}[Log-correlated Gaussian field] \label{def:log-correlated-field}
A distribution-valued centered Gaussian field $X$ in a $d$-torus $\mathbb{T}^d$ is called  \emph{a log-correlated Gaussian field} if it has covariance
\begin{equation}
	\label{eq:general-cov}
	R(y,z):= R_X(y,z) := \E \left[X(y)X(z)\right] = \log \frac{1}{|y-z|_{\T^d}} + g(y,z),
\end{equation}
where we assume that $g \in C^\infty(\mathbb{T}^d \times\mathbb{T}^d)$ and $|\cdot|_{\T^d}$ denotes the distance in the torus. 
\end{definition}
Log-correlated fields on $\T^d$ exist, and a natural way to realize such a field is through the series 
\begin{equation}
\label{eq:field-definition}
X(y) =\sum_{n\in \Z^d} \frac{A_n}{\langle n\rangle^{d/2}}e^{in\cdot y},
\end{equation}
where $A_n$ are mutually independent complex standard Gaussian random variables, and convergence is in the sense of distributions. Recall that above we denote $\langle n \rangle = (1+n^2)^{1/2}$. 
Such a field has covariance (see e.g.~\cite{Oh2020})
$$
R(y-z) = \textnormal{Cov}(X(y),X(z))=-\log |y-z|_{\T^d} + g(y-z),
$$
where $ g\in C^\infty(\T^d)$. Throughout, we consider the exact field given by~\eqref{eq:field-definition}. \textcolor{black}{We note that our results could be extended to cover the case of more general log-correlated fields with covariance given by \eqref{eq:general-cov}, provided that $g(y,z) = g(y-z)$ with suitable assumptions on $g$. However, this would require more detailed analysis on the Fourier coefficients, while in the case when the field is given by \eqref{eq:field-definition}, the Fourier series can be computed more explicitly, see, e.g., proof of Lemma \ref{le:operator-coefficients}}.
\begin{remark}
 Since the covariance is unbounded, the point-wise evaluations $X_\beta(z)$ are not well-defined, and the precise interpretation for the covariance is that it yields  the kernel of the covariance operator:
\begin{align*}
	\E \left[\langle X,\psi_1\rangle \langle X,\psi_2\rangle\right]\; =\; \int_{\mathbb{T}^d\times \mathbb{T}^d}R(y,z)\psi_1(y)\psi_2(z) \, dz \, dy
\end{align*}
for any $\psi_1,\psi_2\in C^\infty (\mathbb{T}^d)$.
\end{remark}
In the sequel, we will also consider a scaling parameter $\beta > 0$, and a log-correlated Gaussian field with scaling parameter $\beta$, $\beta X$, written simply as $X_\beta$. The GMC measure $d\mu_\beta$, depending on a parameter $0<\beta < \sqrt{2d}$, is a random measure constructed from the log-correlated field $X$ formally\footnote{To construct the measure precisely, one can approximate the field, e.g. through mollification and show that the corresponding sequence of random measures converge in probability in the space of Radon measures. \textcolor{black}{Furthermore, under rather mild conditions on the mollifying function, the limiting measure of the sequence of measures does not depend on the mollifier. See \cite{berestycki2017, rhodes2016}.}} as the measure 
\begin{align}\label{eq:chaosmeasure}
    d\mu_\beta(y) = e^{\diamond X_\beta(y)}dy,
\end{align}
where 
\begin{align} \label{eq:wickexp}
	e^{\diamond X_\beta(y)} := \exp\left(X_\beta(y) - \frac{\beta^2}{2}\E \left[X(y)^2\right]\right)
\end{align}
denotes the Wick-exponential. Note that again here the representation is only formal, as we have $\E \left[X(y)^2\right] = \infty$ and the point-wise evaluations $X_\beta(y)$ do not exists. 

In the present paper we shall restrict ourselves to the so-called $L^2$-range, where $\beta <\sqrt{d},$ and this also will be our standing assumption from now on. In this case, by using the notation $\mu_\beta(\psi):=\int_{\mathbb{T}^d}\psi(z)d\mu_\beta(z)$ we obtain for any $\psi\in L^\infty (\mathbb{T}^d)$ \textcolor{black}{and any $\beta, \beta'$ such that $\beta
,\beta' < \sqrt{d}$} that \textcolor{black}{
\begin{align*}
    \mathbb{E}\left[\mu_\beta(\psi)\mu_{\beta'}(\psi)\right]&= \mathbb{E}\left[\int_{\mathbb{T}^d} \psi(y)e^{\diamond X_\beta(y)} \, dy \int_{\mathbb{T}^d} \psi(z)e^{\diamond X_{\beta'}(z)} \, dz\right] \\
    &=\Vert \psi\Vert^2_\infty \int_{\mathbb{T}^d\times\mathbb{T}^d} \mathbb{E} \left[ e^{\diamond X_\beta(y)} e^{\diamond X_{\beta'}(z)} \right] \, dy \, dz \\
    &=
    \Vert \psi\Vert^2_\infty \int_{\mathbb{T}^d\times\mathbb{T}^d} e^{\beta \beta' \EX[X_\beta(y) X_{\beta'}(z)]} \, dy \, dz \\
    &=
    \Vert \psi\Vert^2_\infty \int_{\mathbb{T}^d\times\mathbb{T}^d} e^{\beta \beta' R(y,z)} \,dy \,dz \\
    &=
    \Vert \psi\Vert^2_\infty \int_{\mathbb{T}^d\times\mathbb{T}^d} \abs{y-z}^{-\beta \beta'} e^{\beta \beta' g(y-z)} \,dy \,dz.
\end{align*}
}

which is finite as \textcolor{black}{$\beta \beta' < d$. In particular, this shows that $\mu_\beta(\psi)$ is square-integrable whenever $\beta^2 < d$.}
Note that have used the fact that, for any centered Gaussian random variables $H_1$ and $H_2$, we have
\begin{align*}
	\E \left[e^{\diamond H_1} e^{\diamond H_2}\right] = e^{\E (H_1H_2)},
\end{align*}
and we performed the computation only formally, but as mentioned before, it can be validated by first replacing $X$ by the approximated field $X^{\textcolor{black}{\varepsilon}}$ and letting then $\varepsilon\to 0^+$ (this is essentially as in \cite[Theorem 1.1.]{berestycki2017}).

\subsection{$S$-transforms and log-correlated fields}\label{subsec:Wick}
\label{subsec:S-transform}
In this section we briefly introduce subspaces $L^2_\beta(\Omega)$ that are characterized by (extensions of) $S$-transforms at field points $X_\beta(z)$. We follow the ideas introduced in~\cite{AKNSTV-1} and extend them to the torus $\T^d$ case.
The subspace $L^2_\beta(\Omega)$ is given by the following definition.
\begin{definition}[$L_\beta^2(\Omega)$ space generated by GMC $d\mu_\beta$]\label{def:L2-GMC}
    Let $\beta\in(0,\sqrt{d})$. The space $L_{\beta}^2(\Omega) := L^2(\Omega, \mathbb{P},\mathcal{F}_{d\mu_{\beta}})$ generated by the GMC measure $d\mu_\beta$ is the closed linear space spanned by random variables of the form
    \begin{equation}
        \label{eq:basic-RV}
        \int_{\T^d} \varphi(z) d\mu_\beta(z), \quad \varphi \in C^\infty(\T^d).
    \end{equation}
   \end{definition}
In order to characterize the random variables $Z\in L_\beta^2(\Omega)$ and define the $S$-transform at $X_\beta(s)$, we introduce the operator acting on sufficiently regular functions $\phi:\T^d \to \R$
\begin{align}\label{eq:operator-G}
    G\phi(z) = G_{\alpha}\phi(z): & =\int_{\T^d}\exp (\alpha\EX [X(z)X(y)])\phi(y)\, dy \; \nonumber \\
                                  & =
    \int_{\T^d}|z-y|_{\T^d}^{-\alpha}e^{\alpha g(z-y)}\phi(y) dy,
\end{align}
where $\alpha\in (0,d)$ and $X$ is a log-correlated Gaussian field as in~\cref{def:log-correlated-field}. The motivation for the operator stems from the fact that for random variables 
$Z_1 =\int_{\mathbb{T}^d}\varphi_1(z)d\mu_\beta(z)$ and $Z_2 = \int_{\mathbb{T}^d} \varphi_2(z) d\mu_\beta(z)$, we can (again formally) compute the covariance and get
\begin{equation*}
	\begin{split}
		\E [Z_1 Z_2] &= \E \left[\int_{\mathbb{T}^d} \varphi_1(z)d\mu_\beta(z)\int_{\mathbb{T}^d} \varphi_2(z)d\mu_\beta(z)\right] \\
		&=\E \left[\int_{\mathbb{T}^d}\int_{\mathbb{T}^d} \varphi_1(z)\varphi_2(y)e^{\diamond X_\beta(z)}e^{\diamond X_\beta(y)} dy dz\right] \\
		&= \int_{\mathbb{T}^d}\int_{\mathbb{T}^d}\varphi_1(z)\varphi_2(y) e^{\beta^2 R(z,y)} dy dz \\
		&= \int_{\mathbb{T}^d}\varphi_1(z)G_{\beta^2}\varphi_2(z) dz\\
		&= \langle \varphi_1, G_{\beta^2}\varphi_2\rangle_{L^2(\mathbb{T}^d)}.
	\end{split}
\end{equation*}
As in~\cite{AKNSTV-1}, our characterization is based on the bijectivity properties of the operator $G_\alpha$ that we study next. We begin with the following simple lemma that eventually provides the required decay of the Fourier coefficients.
\begin{lemma}
\label{lem:Fourier-decay}
For $\alpha\in(0,d)$ and a symmetric function $\phi \in C_0^\infty(\R^d)$ with $\phi(0)>0$, let $h: \R^d \to \R$ be given by $h(x) = \phi(x)|x|^{-\alpha}$. Define the Fourier transform of $h$ by 
$$
\widehat{h}(\xi) = \int_{\R^d}h(x)e^{-i\xi \cdot x}dx.
$$
Then there exists constants $c$, $C$, and $K$ such that, for all $|\xi|\geq K$, we have 
$$
c|\xi|^{\alpha-d} \leq \widehat{h}(\xi)\leq C|\xi|^{\alpha-d}.
$$
\end{lemma}
\begin{proof}
Throughout the proof, we denote by $K'$ any generic constant that may change from line to line. We first observe that $\widehat{h}$ is a continuous function. Hence it suffices to prove that there exists $C'>0$ such that
$$
\lim_{|\xi|\to \infty}|\xi|^{d-\alpha}\widehat{h}(\xi) = C'.
$$
By convolution theorem we have $\widehat{h}(\xi) = K'\left(\widehat{\phi} \ast |\cdot|^{\alpha-d}\right)(\xi)$.
Set $R = \sqrt{|\xi|} > 1$. We split the integral over the whole space in the following way:
\begin{align*}
|\xi|^{d-\alpha}\left(\widehat{\phi} \ast |\cdot|^{\alpha-d}\right)(\xi) &= \int_{\R^d} \widehat{\phi}(y)|\xi|^{d-\alpha}|\xi-y|^{\alpha-d}dy\\
&= \int_{B_R(0)}\widehat{\phi}(y)|\xi|^{d-\alpha}|\xi-y|^{\alpha-d}dy \\ 
&\quad + \int_{B^c_R(0) \cap \{|y-\xi|>|\xi|/2 \}}\widehat{\phi}(y)|\xi|^{d-\alpha}|\xi-y|^{\alpha-d}dy  \\
&\quad +\int_{B^c_R(0) \cap \{|y-\xi|<|\xi|/2 \} }\widehat{\phi}(y)|\xi|^{d-\alpha}|\xi-y|^{\alpha-d}dy.
\end{align*}
As $\phi$ is a smooth function, it holds that $|\widehat{\phi}(y)| \lesssim |y|^{-k}$ for any natural number $k$. We apply this to the second and third terms above. For the second term we obtain, since $|y-\xi|>|\xi|/2$,
\begin{align*}
    &\left|\int_{B^c_R(0) \cap \{|y-\xi|>|\xi|/2 \}}\widehat{\phi}(y)|\xi|^{d-\alpha}|\xi-y|^{\alpha-d}dy\right|\\
    &\lesssim \frac{1}{2^{\alpha-d}} \int_{|y|>R} |y|^{-k}dy \\
    &\lesssim  R^{-k+1} 
\end{align*}
that converges to zero for $k>1$ and since $R = \sqrt{|\xi|}\to \infty$. Similarly for the third term, since now $|\widehat{\phi}(y)|\lesssim |y|^{-k} \lesssim |\xi|^{-k}$,
\begin{align*}
    &\left|\int_{B^c_R(0) \cap \{|y-\xi|<|\xi|/2 \}}\widehat{\phi}(y)|\xi|^{d-\alpha}|\xi-y|^{\alpha-d}dy\right|\\
    &\lesssim|\xi|^{-k+d-\alpha}\int_{B^c_R(0) \cap \{|y-\xi|<|\xi|/2 \}}|\xi-y|^{\alpha-d}dy\\
    &\lesssim|\xi|^{-k+d}
\end{align*}
that converges to zero for $k>d$. For the remaining term, we first note that on $B_R(0)$ we have 
$$
|\xi|^{d-\alpha}|\xi-y|^{\alpha-d} \leq K'
$$
for some constant $K'> 0$, since $|y-\xi| \approx |\xi|$ (as $|y| \leq R=\sqrt{|\xi|}$). Hence we apply the Dominated Converge Theorem (as $\widehat{\phi}$ is integrable) to get
$$
\lim_{|\xi|\to \infty}\int_{B_R(0)}\widehat{\phi}(y)|\xi|^{d-\alpha}|\xi-y|^{\alpha-d}dy = \int_{\R^d}\widehat{\phi}(y)dy = (2\pi)^d\phi(0)>0,
$$
which establishes the result.
\end{proof}
\begin{lemma}\label{le:operator-coefficients}
    There exists constants $c$ and $C$ such that the Fourier coefficients of the kernel $H_\alpha = e^{\alpha g(x)}|x|_{\T^d}^{-\alpha}$ with $h \in C^\infty(\T^d)$ satisfy 
    \begin{equation}
    \label{eq:operator-coefficients}
    c\langle n\rangle^{\alpha-d} \leq \widehat{H_\alpha}(n) \leq C\langle n\rangle^{\alpha-d}
    \end{equation}
    for all $n \in \Z^d$. 
\end{lemma}
\begin{proof}
Recall that our field is given by 
\begin{equation}
\label{eq:field}
X(y) =\sum_{n} \frac{A_n}{\langle n\rangle^{d/2}}e^{in\cdot y}
\end{equation}
which has the covariance (see e.g.~\cite{Oh2020})
$$
R(y-z) = \textnormal{Cov}(X(y),X(z))=-\log |y-z|_{\T^d} + g(y-z),
$$
where $g \in C^\infty(\T^d)$. In other words, we are computing the Fourier coefficients of $H_\alpha(x) = e^{\alpha R(x)}$. Note also that it suffices to prove that the coefficients are strictly positive and that~\eqref{eq:operator-coefficients} holds for large enough $|n|\geq N_0$. We begin by showing that the coefficients are strictly positive. Note that it follows from~\eqref{eq:field} that $R$ has Fourier coefficients
$$
\hat{R}(n) = \langle n\rangle^{-d}.
$$
Now
$$
e^{\alpha R(x)} = \sum_{k=0}^\infty \frac{\alpha^k}{k!}R^k(x)
$$
and hence 
$$
\widehat{H_\alpha}(n) = \sum_{k=0}^\infty \frac{\alpha^k}{k!}\widehat{R^k}(n).
$$
Now by the convolution theorem and the fact $\hat{R}(n)>0$ for each $n\in \Z^d$, we obtain $\widehat{H_\alpha}(n) > 0$. It remains to prove that lower and upper bounds in~\eqref{eq:operator-coefficients} hold for large enough $|n|$. To this end, by identifying $H_\alpha$ with a periodic function on $\R^d$, we note that decay of the Fourier coefficients is determined by the singularity $|x|_{\T^d}^{-\alpha}$. Namely, write 
$$
|x|_{\T^d}^{-\alpha}e^{\alpha g(x)} = \rho_1(x)|x|^{-\alpha}+\rho_2(x),
$$
where $\rho_1,\rho_2 \in C^\infty(\T^d)$ and $\rho_1$ satisfies $\rho_1(0)>0$ and $\text{supp}(\rho_1) \subset B\left(0,\frac13\right)$. The Fourier coefficients of $\rho_2$ decay rapidly. In turn, by interpreting the first term as a function in $\R^d$ it is enough to consider its Fourier  transform evaluated at points $\xi = n \in \Z^d$. Now the result follows from Lemma~\ref{lem:Fourier-decay}. This completes the proof. 
\end{proof}
\begin{lemma}\label{le:inversion}
    Let $s\in \R$. Then the  mapping $G_{\alpha}$ given in~\cref{eq:operator-G} extends to a bijective map
    \begin{align*}
        G_\alpha: W^{s+\alpha-d,2}(\T^d) \to W^{s,2}(\T^d).
    \end{align*}
\end{lemma}

\begin{proof}
    Let $s \in \R$. Since $G_\alpha u = H_\alpha \ast u$ is a convolution operator with kernel $H_\alpha$, we have, thanks to convolution theorem and Lemma~\ref{le:operator-coefficients}, that 
\begin{align*}
    \Vert G_\alpha u\Vert_{s,2}^2 & = \sum_{n\in\Z^d}[\widehat{G_\alpha u}(n)]^2 \langle n\rangle^{2s}\\
    &= \sum_{n\in\Z^d}[\widehat{H_\alpha}(n)]^2[\widehat{u}(n)]^2 \langle n\rangle^{2s}\\
    &\approx\sum_{n\in\Z^d}[\widehat{u}(n)]^2 \langle n\rangle^{2s+2\alpha -2d} \\
    &= \Vert u\Vert_{s + \alpha -d,2}^2.
\end{align*}    
This shows that  $G_\alpha : W^{s+\alpha-d,2}(\T^d) \to W^{s,2}(\T^d)$ is bounded and lower bounded. Since the image is obviously dense, it is actually an isomorphism.
This completes the proof.
\end{proof}
Now that the mapping properties of the operator $G_\alpha$ are established, we can characterize random variables in the subspace $L^2_\beta(\Omega)$ and define the $S$-transforms, similar to~\cite{AKNSTV-1}. While the proofs are analogous to the ones presented in~\cite{AKNSTV-1}, for the reader's convenience we present the key steps.
\begin{corollary}
   Let $\beta\in(0,\sqrt{d})$ and suppose $Z \in L^2_\beta(\Omega)$. Then $Z$ has the form
    \begin{equation}
    \label{eq:Z-form}
        Z = \int_{\T^d}\varphi(z)d\mu_\beta(z),
    \end{equation}
    for some $\varphi \in W^{-s_\beta,2}(\T^d)$, where $s_\beta= \frac{d-\beta^2}{2}$.
\end{corollary}
\begin{proof}
    For $Z = \int_{\T^d} \varphi(z) d\mu_\beta(z)$ with $\varphi \in C^\infty(\T^d)$ we have
    \begin{align*}
        \EX [Z^2] & = \int_{\T^d} \int_{\T^d} \varphi(z)\varphi(y)\exp(\beta^2 \EX [X(y)X(z)]) dydz                          \\
                & = \int_{\T^d} \varphi(z)G_{\beta^2}\varphi(z)dz \approx \|\varphi\|^2_{W^{-{s_\beta},2}(\T^d)}
    \end{align*}
    by Lemma~\ref{le:inversion}. The claim now follows by taking the closure in $L^2(\Omega)$.
\end{proof}
\begin{remark}
    The integral in \eqref{eq:Z-form} is understood as the limit in $L^2(\Omega)$ when $\varphi\in W^{-s_\beta,2}(\T^d)$ is approximated with a sequence $\varphi_n \in C^\infty(\T^d)$, and our computation above shows that this limit is independent of the sequence $\varphi_n$. 
\end{remark}
Let us next recall that the concept of $S$-transform of $Z \in L^2(\Omega)$ evaluated at a centered Gaussian random variable $h$ is given by
\begin{align}\label{eq:S-transform}
    (SZ)(h) = \E\left[Ze^{\diamond h}\right],
\end{align}
where
$$
e^{\diamond h} \;:=\; e^{h - \mathbb{E}[h^2 /2]}
$$
is again the Wick exponential. It is well-known that $S$-transform characterizes the random variables in the $L^2(\Omega)$ space of random variables measurable with respect to the sigma-algebra generated by the Gaussian random variables. Indeed, this follows from the fact that such space is spanned by the Wick exponentials $e^{\diamond h}$, see e.g.~\cite{Janson1997}. Moreover, the $S$-transform can be generalized to random elements in the so-called Hida space, in which one can also define the so-called Wick product through the identity
\begin{equation}
\label{eq:wick-product}
(S(Z\diamond Y))(h) = (SZ)(h)\cdot (SY)(h),
\end{equation}
see~\cite{kondratiev1996} for details.

For our purposes, we follow the ideas of our previous article~\cite{AKNSTV-1} and define the concept of $S$-transform also at evaluation points $X_\beta(z)$, characterized by projections onto the subspace $L^2_\beta(\Omega)$. After that, we can define the Wick product and our solution through identity~\eqref{eq:wick-product}. Indeed, an element $V \in L^2_{\beta}(\Omega)$ can be expressed as $\textcolor{black}{\mu_\beta(\psi)} = \int_{\T^d} \psi(z) d\mu_\beta(z)$, where $\psi \in W^{-s_\beta,2}(\T^d)$ with $s_\beta= \frac{d-\beta^2}{2}$, a formal computation gives, for any $Z \in L^2(\Omega)$, that
\begin{equation*}
   \EX [Z\textcolor{black}{\mu_\beta(\psi)}]  = \EX\left[\int_{\T^d} \psi(z)Ze^{\diamond X_\beta(z)}dz\right] =\int_{\T^d} \psi(z)(SZ)(X_\beta(z))dz.
\end{equation*}
Note also that the mapping $\psi \mapsto \EX [Z\textcolor{black}{\mu_\beta(\psi)}]$ is a continuous functional on $W^{-s_\beta,2}(\T^d)$. Hence, by duality, there exists $\iota \in W^{s_\beta,2}(\T^d)$ such that
\begin{equation*}
    \EX [Z\textcolor{black}{\mu_\beta(\psi)}] = \int_{\T^d} \psi(z) \iota(z)dz.
\end{equation*}
This leads to the following definition.
\begin{definition}[$S$-transform in $L^2_\beta(\Omega)$]\label{def:S-transform-log}
    Let $Z \in L^2(\Omega)$ and set $s_\beta= \frac{d-\beta^2}{2}$. For $z\in \T^d$ we identify $(SZ)(X_\beta(z))$ as the (unique) element of $W^{s_\beta,2}(\T^d)$ defined through duality
    \begin{equation*}
        \EX [Z\textcolor{black}{\mu_\beta(\psi)}] = \int_D \psi(z)(SZ)(X_{\beta}(z))dz
    \end{equation*}
    for all $\textcolor{black}{\mu_\beta(\psi)} = \int_{\T^d} \psi(z) d\mu_\beta(z) \in L^2_{\beta}(\Omega)$.
\end{definition}
\begin{remark}
    Note that we have defined $(SZ)(X_\beta(z))$ as a function in $W^{s_\beta,2}(\T^d)$, defined for almost all $z\in \T^d$.
\end{remark}
Our next result shows how $S$-transform is linked to the projections into $L^2_\beta(\Omega)$ (cf.~\cite[Lemma 5.9]{AKNSTV-1}).
\begin{lemma}
    Let $Z \in L^2(\Omega)$ with a decomposition $Z = Z_\beta + Z'_\beta$, where $Z_\beta = \int_{\T^d} \varphi(s)d\mu_\beta(s) \in L^2_\beta(\Omega)$ is the orthogonal projection onto the subspace $L^2_\beta(\Omega) \subset L^2(\Omega)$. Then we have
    \begin{equation*}
        (SZ_\beta)(X_\beta(z)) = \left(G_{\beta^2}\varphi\right)(z)
    \end{equation*}
    and
    \begin{equation*}
        (SZ'_\beta)(X_\beta(z)) = 0.
    \end{equation*}
\end{lemma}
\begin{proof}
    The claim $(SZ'_\beta)(X_\beta(z)) = 0$ follows immediately from~\cref{def:S-transform-log} by observing that then
    \begin{equation*}
        \int_{\T^d} \psi(z)(SZ'_\beta)(X_\beta(z)) dz = 0 \quad \forall \psi \in W^{-s_\beta,2}(\T^d).
    \end{equation*}
    For $Z_\beta$, obtain that
    \begin{equation*}
        \EX [Z_\beta \textcolor{black}{\mu_\beta(\psi)}] = \EX \left[\int_{\T^d} \int_{\T^d} \psi(z)\varphi(y)d \mu_\beta(z)d\mu_\beta(y)\right] = \int_{\T^d} \psi(z) \left(G_{\beta^2}\varphi\right)(z) dz,
    \end{equation*}
    where now, by~\cref{le:inversion}, $\left(G_{\beta^2}\varphi\right)(z) \in W^{s_\beta,2}(\T^d)$. This completes the proof.
\end{proof}
For the proof of the next lemma, see ~\cite[Proof of Lemma 5.10]{AKNSTV-1}.
\begin{lemma}\label{lemma:S-uniqueness}
    The $S$-transform at points $X_\beta(z)$ determine\textcolor{black}{s} random variables $Z \in L^2_\beta(\Omega)$ uniquely.
\end{lemma}
We end this section by the following result allowing us to compute projections of random variables in $L^2_\beta(\T^d)$ into $L^2_{\beta'}(\T^d)$. For an analogous result in the univariate setting, see~\cite[Proposition 5.11]{AKNSTV-1}.
\begin{proposition}\label{prop:projection}
    Let $\beta,\beta' \in (0,\sqrt{d})$ and let
    \begin{equation*}
        Z = \int_{\T^d} \varphi(z)d\mu_\beta(z) \in L^2_\beta(\Omega)
    \end{equation*}
    for some $\varphi \in W^{-\frac{d-\beta^2}{2},2}(\T^d)$. Denote the projection of $Z$ into the subspace $L_{\beta'}^2(\Omega)$ by
    \begin{equation*}
        Z_{\beta'} = \int_{\T^d} \varphi_{\beta'}(z)d\mu_{\beta'}(z).
    \end{equation*}
    The function $\varphi_{\beta'} \in W^{-\frac{d-(\beta')^2}{2},2}(\T^d)$ solves
    $G_{\beta\beta'}\varphi = G_{(\beta')^2}\varphi_{\beta'}$.
\end{proposition}
\begin{proof}
    Let $Z = \int_{\T^d} \varphi(z)d\mu_\beta(z) \in L^2_\beta(\Omega)$ and $\textcolor{black}{\mu_\beta(\psi)} = \int_{\T^d} \psi(z)d\mu_{\beta'}(z) \in L^2_{\beta'}(\Omega)$ be a test random variable. It is enough to check that $\EX [Z_{\beta'}\textcolor{black}{\mu_\beta(\psi)}] = \EX [Z\textcolor{black}{\mu_\beta(\psi)}]$ for each $\psi \in W^{-\frac{d-(\beta')^2}{2},2}(\T^d)$. Here we have
    \begin{equation*}
        \EX [Z_{\beta'}\textcolor{black}{\mu_\beta(\psi)}] = \int_{\T^d} \psi(z)G_{(\beta')^2}\varphi_{\beta'}(z)dz.
    \end{equation*}
    Similarly,~\cref{def:S-transform-log} yields
    \begin{equation*}
        \EX [Z\textcolor{black}{\mu_\beta(\psi)}] = \int_{\T^d} \psi(z)(SZ)(X_{\beta'}(z))dz 
    \end{equation*}
    for $(SZ)(X_{\beta'}(z)) \in  W^{\frac{d-(\beta')^2}{2},2}(\T^d)$ while formal computations as above gives, assuming \eqref{eq:claim_projection} is true,
    \begin{equation}
        \label{eq:beta'-projection}
        \EX [Z\textcolor{black}{\mu_\beta(\psi)}] = \int_{\T^d} \psi(z)G_{\beta\beta'}\varphi(z)dz,
    \end{equation}
    where $\psi \in W^{-\frac{d-(\beta')^2}{2},2}(\T^d)$ and $\varphi \in  W^{-\frac{d-\beta^2}{2},2}(\T^d)$. 
    We claim that indeed
    \begin{align}\label{eq:claim_projection}
        G_{\beta\beta'}\varphi(z) \in W^{\frac{d-(\beta')^2}{2},2}(\T^d).      
    \end{align}
    Using \cref{eq:claim_projection}, we see that~\cref{eq:beta'-projection} follows, since from~\cref{def:S-transform-log}, the uniqueness of the $S$-transform gives $(SZ)(X_{\beta'}(z)) = G_{\beta\beta'}\varphi(z)$. Now~\cref{le:inversion} gives that $G_{(\beta')^2}:W^{-\frac{d-(\beta')^2}{2},2}(\T^d)\to W^{\frac{d-(\beta')^2}{2},2}(\T^d)$ is invertible and hence $\EX [Z_{\beta'} \textcolor{black}{\mu_\beta(\psi)}] = \EX [Z \textcolor{black}{\mu_\beta(\psi)}]$ holds if and only if $G_{\beta \beta'} \varphi = G_{(\beta')^2} \varphi_{\beta'}$. For the claim~\cref{eq:claim_projection} note that, by Lemma~\ref{le:inversion}, we have $G_{\beta\beta'}:W^{s + \beta\beta' - d,2}(\T^d) \to W^{s,2}(\T^d)$ for any $s\in \R$. In particular, this holds for $s = \frac{d}{2} - \frac{(\beta')^2}{2}$ yielding 
    $G_{\beta\beta'}:W^{-\frac{d}{2} - \frac{(\beta')^2}{2}+ \beta\beta',2}(\T^d) \to W^{\frac{d}{2} - \frac{(\beta')^2}{2},2}(\T^d)$. Now \cref{eq:claim_projection} follows by noting that $W^{-\frac{d-\beta^2}{2}}(\T^d) \subset W^{-\frac{d}{2} - \frac{(\beta')^2}{2}+ \beta\beta',2}(\T^d)$ since
    $$
    -\frac{d}{2} - \frac{(\beta')^2}{2}+ \beta\beta' \leq -\frac{d-\beta^2}{2}
    $$
    for arbitrary $\beta,\beta'\in \R$.  Hence,~\cref{eq:claim_projection} is proved.

\end{proof}

\section{Main result}\label{sec:mainresult}

In this section, we first present a heuristic computation which motivates the definition of the solution of our problem. Indeed, consider formally the stochastic PDE
\begin{equation} \label{eq:stochastic-pde}
	\begin{aligned}
		-\nabla \cdot (e^{ \diamond (-X_\beta)} \diamond \nabla U) & =\nabla \cdot (e^{\diamond (-X_\beta)} \diamond \mathbf{F}), \quad \textnormal{in } \mathbb{T}^d \\
		y \mapsto U(y) \textnormal{ is } \mathbb{T}^d              & \textnormal{-periodic},
	\end{aligned}
\end{equation}
where $X_\beta$ is the log-correlated Gaussian field with scaling parameter $\beta < \sqrt{d}$, and $\mathbf{F}=(F_1, \dots, F_d)$ is a random element with real-valued random variables as its components.

Now fix $y \in \mathbb{T}^d$, and take the $S$-transform of~\eqref{eq:stochastic-pde} on Gaussian  $Z$, which does not depend on $y$. Hence the equation becomes, using Fubini's Theorem and $S(X \diamond Y)(Z) = (SX)(Z)\cdot (SY)(Z)$,
\begin{align} \label{eq:pde-s-trans}
- \nabla \cdot (Se^{ \diamond (-X_\beta(y))}(Z) \cdot (S\nabla U(y))(Z)) = \nabla \cdot ( (Se^{ \diamond(- X_\beta(y))})(Z) \cdot (S\mathbf{F}(y))(Z) ) 
\end{align}
 Now, recall that since
\begin{align*}
	e^{\diamond (X+Y)} = e^{-\E XY} e^{\diamond X} e^{\diamond Y},
\end{align*}
and therefore $ (Se^{ \diamond (-X_\beta(y))})(Z) = e^{-\EX \left[ X_\beta(y) Z \right] }$, so that~\eqref{eq:pde-s-trans} becomes
\begin{align*}
	-\nabla \cdot (e^{-\E \left[ X_{\beta}(y) Z \right]}     (S\nabla U(y))(Z))= \nabla \cdot (e^{-\E \left[ X_{\beta}(y) Z \right]}(S\mathbf{F}(y))(Z)).
\end{align*}
Now, choosing $Z = X_\beta(z)$, for fixed $z \in \T^d$, this becomes
\begin{align} \label{eq:pde-s-transformed}
	-\nabla \cdot (e^{-\beta^2 R(y;z)}  \nabla u(y;z)   ) = \nabla \cdot (e^{-\beta^2 R(y;z)} \mathbf{f}(y;z)),
\end{align}
where we have denoted  $R(y;z) :=\E [X(y) X(z)]$, and $u(y;z) = (S U(y))(X_\beta(z))) = \E[ U(y) e^{\diamond X_\beta(z)}]$, which is well-defined by Definition~\ref{def:S-transform-log}, and $\mathbf{f}(y;z) = (S\mathbf{F}(y))(X_\beta(z))$.

Hence, we see that the problem in~\eqref{eq:stochastic-pde} has a deterministic counterpart on the $S$-transform side, parameterized by $z$, given by
\begin{equation} \label{eq:pde-deterministic-problem}
	\begin{aligned}
		-\nabla \cdot (w(y;z) \nabla u(y;z)   ) & = \nabla \cdot (w(y;z) \mathbf{f}(y;z)), \quad \textnormal{in } \mathbb{T}^d \\
		y \mapsto u(y;z) \textnormal{ is } \mathbb{T}^d & \textnormal{-periodic},
	\end{aligned}
\end{equation}
where 
$$
w_z(y):=w(y;z) = \abs{y-z}^{\beta^2}e^{-\beta^2 g(y-z)}
$$
with $g$ smooth and bounded. 
\begin{definition}
We say that \eqref{eq:pde-deterministic-problem} is solvable if, for a fixed $z\in \T^d$, there exists a unique (up to constant) function $u(\cdot;z) \in W^{1,2}(\T^d;w_z)$, where $W^{1,2}(\T^d;w_z)$ denotes the weighted Sobolev space (see \eqref{eq:weighted-Sobolev}), that satisfies \eqref{eq:pde-deterministic-problem}. 
\end{definition}

Based on the previous computation, we consider solutions of~\eqref{eq:stochastic-pde} in the following way:

\begin{definition}\label{def:pde-solution}
	Let $X_\beta$ be a a log-correlated Gaussian field with scaling parameter $0 < \beta < \sqrt{d}$. Assume that $\mathbf{f}(y;z) = (f_1(y;z),\ldots,f_d(y;z))$ satisfies, for every $z\in \T^d$ and $j=1,\ldots,d$, $f_j(\cdot;z)\in L^2(\T^d)$. We say that~\eqref{eq:stochastic-pde}  is solvable in the space $L_\beta^2(\Omega)$, whenever the solution of the $S$-transformed, deterministic PDE, for all $z \in \T^d$,
	\begin{equation}\label{eq:s-transformed-pde}
		\begin{aligned}
			-\nabla_y \cdot (w(y;z) \nabla_y u(y;z)   ) & = \nabla_y \cdot (w(y;z) \mathbf{f}(y;z)), \quad \textnormal{in } \mathbb{T}^d        \\
			y \mapsto u(y;z) \textnormal{ is } \mathbb{T}^d                           & \textnormal{-periodic},
		\end{aligned}
	\end{equation}
	satisfies, for any fixed $y\in \T^d$, $u(y;\cdot)\in W^{s_\beta,2}(\T^d)$.
 
\end{definition}
We note the asymmetric role of the variables $y$ and $z$ above in Definition~\ref{def:pde-solution}; while the equation, parameterized by $z$, contains derivatives only in the $y$-variable, in order to guarantee the existence of the inverse $G_{\beta^2}^{-1} u$, one needs regularity from $u$ in the $z$-variable, see Lemma~\ref{le:inversion}. 

We note that unpacking the definitions of the space $L_\beta^2(\Omega)$, the $S$-transform, and the operator $G$ in \cref{def:L2-GMC,def:S-transform-log} and equation \cref{eq:operator-G} respectively, \cref{def:pde-solution} yields us a representation of the solution as
\begin{align*}
	U(y) = \int_{\mathbb{T}^d} G_{\beta^2}^{-1}u(y;z)) d\mu_\beta(z) = \int_{\mathbb{T}^d} \varphi(y;z) d\mu_\beta(z), \qquad x \in \mathbb{T}^d
\end{align*}
where $\varphi \in W^{-s_\beta,2}(\mathbb{T}^d)$, with $s_\beta = \frac{d-\beta^2}{2}$, and $d\mu_\beta$ is a Gaussian multiplicative chaos measure with parameter $\beta < \sqrt{d}$. Moreover, \cref{lemma:S-uniqueness} guarantees uniqueness in this representation class -- although we remark that we have uniqueness only in the subspace $L_\beta^2(\Omega)$, not in the entirety of the ambient space $L^2(\Omega)$.

Next, we state our main result concerning the solvability of the periodic problem in~\eqref{eq:stochastic-pde}:
\begin{theorem}\label{thm:main-result}
	Let $X_\beta$ be a log-correlated Gaussian field with scaling parameter $0 < \beta < \sqrt{d}$, and let $\mathbf{F} \in \left(L^2(\Omega)\right)^d$ with $\mathbf{f}(y;z) := (S\mathbf{F}(y))(X_\beta(z))$ satisfying
 \begin{equation*}
		\sum_{0 \leq |\alpha_1+\alpha_2| \leq \lfloor \frac{d+1}{2} \rfloor} \norm{\partial_{y}^{\alpha_1} \partial_z^{\alpha_2} \mathbf f}_{L_z^p(\T^d;L_y^q(\T^d;w_z))} < \infty,
	\end{equation*}
	for some $q > 2d-\eps$ and $p > d/\gamma$, where $\gamma = \gamma(\beta,d,q) \in (0,1)$ and $\eps=\eps(\beta,d)>0$.
 
 Then the problem~\eqref{eq:stochastic-pde} admits a solution in the sense of Definition \ref{def:pde-solution}. In particular, there exists a function  $z \mapsto \varphi(y;z) = (G_{\beta^2}^{-1}u(y;\cdot))(z) \in W^{-s_\beta,2}(\mathbb{T}^d)$ with $s_\beta = \frac{d-\beta^2}{2}$ such that
	\begin{align*}
		U(y) = \int_{\mathbb{T}^d} \varphi(y;z) d\mu_\beta(z), \qquad y \in \mathbb{T}^d.
	\end{align*}
	Above $d\mu_\beta$ is a Gaussian multiplicative chaos measure with parameter $\beta$.
\end{theorem}

\begin{remark}
    Before giving the proof of Theorem \ref{thm:main-result}, we make some remarks on the right-hand side $\mathbf{f}$:
    \begin{itemize}
        \item In the simple case when $\mathbf{F}$ is deterministic (hence $\mathbf{f}$ is independent of the $z$-variable) and has sufficient regularity in the $y$-variable, the theorem guarantees the existence of a solution.
        \item More generally, consider now some non-deterministic $\mathbf{F}$, which belongs to $\left(L_\beta^2(\Omega)\right)^d$ (or its orthogonal projection), that is, $\mathbf{F}$ is of the form
        \begin{align*}
            \mathbf{F}(y) = \int_{\T^d} \varphi(y,z) d\mu_\beta(z),
        \end{align*}
        for some function $\varphi(y,\cdot) \in \left(W^{\eta,2}(\T^d)\right)^d$. Hence $(S\mathbf{F})(X_\beta(z)) = G_{\beta^2}(\varphi(y, \cdot))(z)$, which has mapping properties to $\left(W^{\gamma,2}(\T^d)\right)^d$ given by Lemma \ref{le:inversion}. Hence provided sufficient regularity in both $y$- and $z$-variables for $\varphi(y,z)$, the theorem guarantees the existence of a solution.
    \end{itemize}
\end{remark}

\begin{proof}[Proof of Theorem \ref{thm:main-result}]
	We start by noting that it is a general fact of weighted degenerate elliptic partial differential equations that~\eqref{eq:s-transformed-pde} admits a solution $u$ in the weighted Sobolev space $W^{1,2}(\mathbb{T}^d; w_z)$, where $w(y;z)=\abs{y-z}^{\beta^2}e^{-\beta^2 g(y-z)}$~\cite[Theorem 2.2.]{FKS}. In order for $u$ to satisfy Definition~\ref{def:pde-solution}, we need sufficient regularity for $u$ in the $z$-variable. Indeed, we show that $z \mapsto u(y;z) \in W^{s_\beta,2}(\mathbb{T}^d)$ uniformly in $y$, where $s_\beta = \frac{d-\beta^2}{2}$, and then invert via the operator $G$ in Lemma~\ref{le:inversion}.

	By Theorem~\ref{t.torus.reg} below we have that $z \mapsto u(y;z) \in W^{d/2,2}(\mathbb{T}^d)$ uniformly in $y$. Recall that the space $W^{d/2,2}(\mathbb{T}^d)$ embeds into $W^{s_\beta,2}(\mathbb{T}^d)$ continuously (see, for example~\cite[Theorem 1, p.82]{RunstSickel1996}). By Lemma~\ref{le:inversion} the operator $G_{\beta^2}$ is a bijection between $W^{-s_\beta,2}(\mathbb{T}^d)$ and $W^{s_\beta,2}(\mathbb{T}^d)$, and therefore there exists a unique element $\varphi$ of $W^{-s_\beta,2}(\mathbb{T}^d)$ such that $(G_{\beta^2}^{-1} u(y;\cdot))(z) = \varphi(y;z)$ for all $y \in \mathbb{T}^d$; and hence
	\begin{align*}
		U(y) = \int_{\mathbb{T}^d} \varphi(y;z) d\mu_\beta(z), \qquad y \in \mathbb{T}^d
	\end{align*}
	is a solution of~\eqref{eq:stochastic-pde}, where $d\mu_\beta$ is the GMC measure with parameter $\beta < \sqrt{d}$ on $\mathbb{T}^d$.
	Finally, uniqueness of the solution in the space $L_\beta^2(\Omega)$ is immediate from the bijectivity of $G_{\beta^2}$. This completes the proof.
\end{proof}

\section{Regularity of solutions of the deterministic PDE}
\label{sec:regularity}

In this section, we prove the required regularity result for the deterministic equation, \cref{t.torus.reg}, used in the proof of our main result. Unless stated otherwise,
throughout we let $w_z(\cdot):=w(\cdot;z)$ denote a periodic function in $\T^d$ that can be written as $w(x;z) = w(x-z)$ for $x,z \in \T^d$, satisfying $w(x) \sim |x|^{\beta^2}$ for $x \in \T^d$, for some $0 < \beta^2 < d$. In particular, this is the case for the weight $w_z(y)=|y-z|^{\beta^2}e^{-\beta^2 g(y-z)}$. Here and in the sequel, we identify $\T^d$ with $[-1,1]^d$ with periodic boundary conditions. \textcolor{black}{As our distance we use the distance introduced in \cref{subsec:notation}, that behaves as Euclidean distance on small scales. As in what follows we are essentially interested only in local behavior around the singular pole, for notational simplicity we work only with the Euclidean distance.}

We first remark that the weight $\widehat{w} = |x|^{\beta^2}$ is a Muckenhoupt-$A_2$ (see \cite{FKS} and \cref{a.muckenhoupt}). Specifically the estimate
\begin{align*}
	\sup_B \frac{1}{|B|}\int_{B} \widehat{w} dx \left (\frac{1}{|B|} \int_B \widehat{w}^{-1} dx\right ) \leq \frac{1}{d^2-\beta^4}
\end{align*}
holds true, where the supremum is taken over all Euclidean balls. Although we will not use this form of the estimate, it is worth noting that $\widehat{w}$ is a Muckenhoupt-$A_2$ exactly when $\beta^2 < d$. Let $A_2(w)$ denote the $A_2$ constant (see~\cite{FKS}). Then from the above estimate we obtain $A_2(w) \leq \frac{C(w)}{d^2-\beta^4}$.

Let next $z \in \T^d$ be fixed and recall that the weighted Sobolev space $W^{m,p}(\Omega;w_z)$ for $1 \leq p < \infty$ is given by the completion of  $C^\infty(\Omega)$ under
\begin{align*}
	\|u\|_{W^{m,p}(\Omega;w_z)}
	:= \|u\|_{L^p(\Omega;w_z)} + \sum_{1 \leq |\alpha| \leq m} \|\partial_x^\alpha u\|_{L^p(\Omega;w_z)}.
\end{align*}
In the sequel, we use notation $d\textcolor{black}{\nu}_z = w(x;z) dx$ and if there is no danger of confusion, we will suppress the $z$ and write simply $w$ or $d\textcolor{black}{\nu}$ instead. 

For our proofs, we also define the Euclidean balls with center at $x$ as $B_r(x) := \{y \in \R^d : |y-x| < r\}$, and $B_r := B_r(0)$ for short. Finally, for an $A_2$ weight $w$ we define the mass of a set $E \subset \R^d$ as $w(E) = \int_E w(x) dx$. The weighted normalized norms for balls $B_r(x) \subset \T^d$ as
\begin{align*}
	\|u\|_{\underline{W}^{m,p}(B_r(x);\, w_z)}
	:=
	\sum_{0 \leq |\alpha| \leq m} \left ( \fint_{B_r(x)} |\partial^\alpha u|^p d\textcolor{black}{\nu}_z \right )^{\frac{1}{p}},
\end{align*}
where the average in the integral is with respect to the measure $d\textcolor{black}{\nu}_z$.
The norm $\underline{L}^p$ is defined similarly. 


Our main result of this section is the following regularity result. 
\begin{theorem} \label{t.torus.reg}
	Let $u(\cdot;z) \in H^{1}(\T^d;w)$ be a weak solution to the equation
	\begin{align*}
		\begin{cases}
			- \nabla \cdot (w(\cdot; z) \nabla u(\cdot; z)) = \nabla \cdot (w(\cdot; z) \mathbf{f}(\cdot;z)), & \text{ in $\T^d$,} \\
			\int_{\T^d} u(x;z) dx = 0,
		\end{cases}
	\end{align*}
	for every $z \in \T^d$.
	Then, there exists $\eps = \eps(\beta,d)>0$ such that if 
	\begin{equation*}
		\sum_{0 \leq |\alpha_1+\alpha_2| \leq \lfloor \frac{d+1}{2} \rfloor} \norm{\partial_{x}^{\alpha_1} \partial_z^{\alpha_2} \mathbf f}_{L_z^p(\T^d;L_x^q(\T^d;w_z))} < \infty,
	\end{equation*}
	for $q > 2d-\eps$ and $p > d/\gamma$ where $\gamma = \gamma(\beta,d,q) \in (0,1)$. Then 
	\begin{align}\label{e.regularity}
		u(x;\cdot) \in H^{d/2}(\T^d),
	\end{align}
	uniformly in $x$.
\end{theorem}

We will prove this theorem by using regularity theory for weighted equations, see~\cite{FJK,FKS}, which we record in the next subsection.

Note that from $w\in L^\infty(\T^d)$, it follows that $W^{d/2,q}(\T^d;dx)\subset W^{d/2,q}(\T^d;w)$ by basic estimates as in~\cite[Lemma 1.13]{HKM}.

\subsection{Basic regularity and Green function bounds}

\begin{lemma} \label{l.holder}
	Let $w$ be an $A_2$ weight and let $u \in H^1(B_1;w)$ be a weak solution in $B_1$ of
	\begin{equation*}
		-\nabla \cdot (w \nabla u) = \nabla \cdot (w \mathbf f).
	\end{equation*}
	There exists $\eps = \eps(A_2(w)) > 0$ such that if $f \in L^q(B_1, \R^d;w)$ for $q > 2d-\eps$, then $u$ is Hölder continuous. Specifically, for $\varrho < 1$ there exists a constant $C = C(A_2(w),d) > 1$ and an exponent $\alpha = \alpha(A_2(w),d) \in (0,1)$ such that
	\begin{equation}\label{e.holder.1}
		\sup_{B_\varrho} |u| + \frac{\osc_{B_\varrho} u}{\varrho^\alpha} \leq C \left ( \norm{u}_{\underline{L}^2(B_1;w)} + \norm{\mathbf f}_{\underline{L}^q(B_1;w)} \right ).
	\end{equation}
	Furthermore, it holds
	\begin{equation}\label{e.holder.2}
		\fint_{B_\varrho} |\nabla u|^2 d\textcolor{black}{\nu} 
		\lesssim 
		\varrho^{2\alpha-2} \left (\norm{\mathbf f}^2_{L^q(B_1;w)} + \fint_{B_1} |\nabla u|^2 d\textcolor{black}{\nu} \right ),
	\end{equation}
	where $d\textcolor{black}{\nu} = w dx$.
\end{lemma}
\begin{proof}
	In the following we assume without loss of generality that $\int_{B_1} u d\textcolor{black}{\nu} = 0$, and we note that estimate \cref{e.holder.1} is just~\cite[Theorem 2.3.15]{FKS}.

	To prove \cref{e.holder.2}, note that \cref{e.holder.1} and the weighted Poincaré inequality, see~\cite{FKS}, gives that for $|x| < 1/2$ we have
	\begin{align} \label{e.holder1}
		\notag
		|u(x) - u(0)|^2
		 & \lesssim
		|x|^{2\alpha} \left ( \|u\|^2_{\underline{L}^2(B_1;w)} + \norm{\mathbf f}^2_{\underline{L}^q(B_1;w)} \right )
		\\
		 & \lesssim
		|x|^{2\alpha} \left ( \|\nabla u\|^2_{\underline{L}^2(B_1;w)} + \norm{\mathbf f}^2_{\underline{L}^q(B_1;w)} \right ).
	\end{align}
	Now take a cutoff function $\chi_{B_\varrho} \leq \phi \leq \chi_{B_{2\varrho}}$ with $\abs{\nabla \phi} \lesssim \frac{1}{\varrho}$, and test the equation with $\psi = (u(x) - u(0))\phi^2$. We get
	\begin{multline*}
		2\int  \nabla u \cdot \nabla  \phi\cdot \phi (u(x)-u(0)) d\textcolor{black}{\nu} + \int |\nabla u|^2 \phi^2 d\textcolor{black}{\nu}
		\\
		=
		2\int \mathbf f \cdot \nabla \phi \cdot \phi (u(x) - u(0)) d\textcolor{black}{\nu} + \int \mathbf f \cdot \nabla u \phi^2 d\textcolor{black}{\nu}
		.
	\end{multline*}
	Using Young's inequality for products, we get the classical Caccioppoli estimate
	\begin{equation*}
		\int |\nabla u|^2 \phi^2 d\textcolor{black}{\nu}
		\lesssim
		\int \abs{\mathbf f}^2 \phi^2 d\textcolor{black}{\nu}
		+\int |\nabla \phi|^2 (u(x) - u(0))^2 d\textcolor{black}{\nu}.
	\end{equation*}
	Now
	\begin{equation*}
		\int |u(x) - u(0)|^2 |\nabla \phi|^2 d\textcolor{black}{\nu}
		\lesssim \frac{\sup_{x \in B_{2\varrho}} |u(x) - u(0)|^2}{\varrho^2} w(B_{2\varrho}),
	\end{equation*}
	and, using Hölder's inequality,
	\begin{equation*}
		\int_{B_\varrho} \abs{\mathbf f}^2 d\textcolor{black}{\nu} \leq w(B_\varrho)^{\frac{q-2}{q}} \left ( \int_{B_\varrho} \abs{\mathbf f}^q d\textcolor{black}{\nu} \right)^{2/q}.
	\end{equation*}
	Consequently, 
	\begin{equation*}
		\fint_{B_\varrho} |\nabla u|^2 d\textcolor{black}{\nu}
		\lesssim
		w(B_{2\varrho})^{\frac{q-2}{q}-1} \left ( \int_{B_{2\varrho}} \abs{\mathbf f}^q d\textcolor{black}{\nu} \right)^{2/q}
		+\frac{\sup_{x \in B_{2\varrho}} |u(x) - u(0)|^2}{\varrho^2}.
	\end{equation*}
	Finally, by \cref{e.holder1} there exists another $\gamma = \gamma(A_2(w),d) \in (0,1)$ such that
	\begin{equation*}
		\fint_{B_\varrho} |\nabla u|^2 d\textcolor{black}{\nu} \lesssim \varrho^{2\gamma-2} \left (\norm{\mathbf f}^2_{L^q(B_1;w)} + \fint_{B_1} |\nabla u|^2 d\textcolor{black}{\nu} \right ).
	\end{equation*}
	This completes the proof.
\end{proof}

If we already know a bound on the $\norm{\mathbf f}_{\underline{L}^2(B_\varrho;w)}$ norm in terms of powers of $\varrho$, we can bypass the $L^q$ norm, and obtain the same regularity as in \cref{l.holder}:

\begin{lemma} \label{l.holder2}
	Let $w$ be an $A_2$ weight and let $u \in H^1(B_1;w)$ be a weak solution to
	\begin{equation} \label{e.holder2.eq}
		\nabla \cdot (w \nabla u) = \nabla \cdot (w \bf{f})
	\end{equation}
	in $B_1$. For all $\rho < 1$, let $f \in L^2(B_1,\R^d;w)$ satisfy the growth estimate
	\begin{equation}\label{e.fdecay}
		\fint_{B_\varrho} \abs{\mathbf f}^2 d\textcolor{black}{\nu} \lesssim M \varrho^{2\gamma-2} \fint_{B_1} \abs{\mathbf f}^2 d\textcolor{black}{\nu}
	\end{equation}
	for some exponent $\gamma = \gamma(A_2(w), d) \in (0,1)$ and $d\textcolor{black}{\nu} = w dx$. Then there exists an exponent $\widehat{\gamma} = \widehat{\gamma}(A_2(w),d,\gamma) \in (0,1)$ and a constant $C = (A_2(w),d,\gamma) \geq 1$ such that
	\begin{equation*}
		\fint_{B_\varrho} |\nabla u|^2 d\textcolor{black}{\nu} \leq C \varrho^{2\widehat{\gamma}-2} \left ( \fint_{B_1} |\nabla u|^2 d\textcolor{black}{\nu}  + M \fint_{B_1} \abs{\mathbf f}^2 d\textcolor{black}{\nu}\right ).
	\end{equation*}
\end{lemma}
\begin{proof}
	Let $\varrho < R \leq 1$ and construct the comparison solution $v$ that solves
	\begin{equation*}
		\begin{cases}
			\nabla \cdot (w \nabla v) = 0, & \text{ in $B_R$.} \\
			v = u,                       & \partial B_R.
		\end{cases}
	\end{equation*}
	Denote $\widehat{u} = u - v$. Then $\widehat{u}$ solves \cref{e.holder2.eq} with $\widehat{u} \in H^1_0(B_R;w)$; thus we can test \cref{e.holder2.eq} with $\widehat{u}$ itself, which gives after applying Young's inequality together with the Poincaré inequality, that
	\begin{equation*}
		\int_{B_R} |\nabla \widehat{u}|^2 d\textcolor{black}{\nu} \lesssim \int_{B_R}  \abs{\mathbf f}^2 d\textcolor{black}{\nu}.
	\end{equation*}
	Also, by direct computation and applying \cref{l.holder} we get
	\begin{align*}
		\int_{B_\varrho} |\nabla u|^2 d\textcolor{black}{\nu}
		 & \lesssim
		\int_{B_\varrho} |\nabla v|^2 d\textcolor{black}{\nu} + \int_{B_\varrho} |\nabla \widehat{u}|^2 d\textcolor{black}{\nu}
		\\
		 & \lesssim
		\frac{w(B_\varrho)}{w(B_R)}\left ( \frac{\varrho}{R}  \right )^{2\alpha-2} \int_{B_R} |\nabla v|^2 d\textcolor{black}{\nu} + \int_{B_R} |\nabla \widehat{u}|^2 d\textcolor{black}{\nu}
		\\
		 & \lesssim
		\frac{w(B_\varrho)}{w(B_R)}\left ( \frac{\varrho}{R}  \right )^{2\alpha-2} \int_{B_R} |\nabla u|^2 d\textcolor{black}{\nu} + \int_{B_R} \abs{\mathbf f}^2 d\textcolor{black}{\nu}.
	\end{align*}
	Now by \cref{e.fdecay} we have by a standard iteration inequality, see for instance~\cite[Lemma 3.4]{HanLin}, that
	\begin{equation*}
		\fint_{B_\varrho} |\nabla u|^2 d\textcolor{black}{\nu} \lesssim \left ( \frac{\varrho}{R}\right )^{2\widehat{\alpha} - 2} \fint_{B_R} |\nabla u|^2 d\textcolor{black}{\nu} + \varrho^{2\widehat{\alpha} - 2} M \fint_{B_R} \abs{\mathbf f}^2 d\textcolor{black}{\nu},
	\end{equation*}
	which completes the proof.
\end{proof}
We finally end with a regularity result for the polynomially bounded weight $w_z$ away from the pole $z$. We will later use this to derive bounds for the Green function for polynomially weighted equations.

\begin{lemma}\label{l.grad.nopole}
	Let $z^\ast \in \R^d$, $x^\ast \in \R^d \setminus \{z\textcolor{black}{^\ast}\}$, and let
	\begin{align}  \label{e.r.def}
		r = \frac{\min \left( |x^\ast-z^\ast| , 1 \right)}{16}
		.
	\end{align}
	For $w_{z^\ast}(x) = |x-z^\ast|^{\beta^2}$,
	let $u \in H^1(B_{8r}(x^*); w_{z^\ast})$ be a solution of
	\begin{align} \notag 
		\nabla \cdot (w_{z^\ast}\nabla u ) = 0 \quad \mbox{ in }  B_{8r}(x^*).
	\end{align}
	Then there exist a constant $C(\beta,d)<\infty$ such that
	\begin{align}  \label{e.grad.local}
		\sup_{B_r(x^\ast)}|\nabla u(x)|
		 & \leq
		C r^{-1} \inf_{a} \fint_{B_{8r}(x^*)} |u-a|^2 d\textcolor{black}{\nu}_{z^\ast}.
	\end{align}
\end{lemma}

\begin{proof}
	Since for any $y \in B_{8r}(x^\ast)$ we have $|y-z^\ast| \geq 8r$, it follows that, for any $i = 1,\ldots,d$ and for $\widehat{w} = \frac{w_{z^\ast}}{\inf_{B_{8r}(x^\ast)} w_{z^\ast}}$, we have
	\begin{align*}
		\sup_{B_{8r}(x^\ast)} \widehat{w} \leq C(\beta), \quad \text{and}, \quad \sup_{B_{8r}(x^\ast)} |\partial_{x_i} \widehat{w}| \leq \frac{C(\beta)}{r}.
	\end{align*}
	As such, $v = u - a$ solves the equation $\nabla \cdot (\widehat{w} \nabla v)= 0$, where the coefficient $\widehat{w}$ is Lipschitz continuous with Lipschitz constant $C/r$ and $\widehat{w}\geq 1$.
	Now the regularity theory for such equations, see display (3.17) in~\cite{HanLin}, implies that
	\begin{align*}
		\sup_{x \in B_r(x^\ast)} |\nabla v| \lesssim \frac{1}{r} \fint_{B_{8r}(x^\ast)} v^2 d\textcolor{black}{\nu}_{z^\ast},
	\end{align*}
	proving the claim.
\end{proof}

\begin{lemma} \label{l.green}
	There exists a periodic Green function $(x,y,z) \mapsto G(x,y\,; z)$ solving, for every $z,y \in \T^d$, the equation 
	\begin{align}\label{e.green.eq}
		\left\{
		\begin{aligned}
			 & -\nabla \cdot \left( w(\cdot\,;z) \nabla  G(\cdot,y\,;z) \right)  =  \delta(\cdot - y)-\textcolor{black}{{\frac{1}{\int_{\T^d}dz}}} & \mbox{in} & \ \T^d, \\
			 & \int_{\T^d} G(x,y\,;z) dx = 0.
		\end{aligned}
		\right.
	\end{align}
	Moreover, the Green function is symmetric in the first two variables, that is, for every $x,y,z \in \T^d$,
	\begin{align}  \label{e.green.sym}
		G(x,y\,; z) = G(y, x\,; z)\,.
	\end{align}
	Furthermore, there exists a constant $C(\beta,d) < \infty$ such that, for every $x,y,z \in \T^d$, it holds for $i=0,1$ that
	\begin{align}\label{e.green.bdd}
		(|x-y| \wedge |y-z|)^i |\nabla_y^i G(x,y;z)| \lesssim
		\begin{cases}
			(|x-z| \vee |y-z|)^{-\beta^2} |x-y|^{2-d},          & d > 2, \\
			(|x-z| \vee |y-z|)^{-\beta^2} \log \frac{C}{|x-y|}, & d=2.
		\end{cases}
	\end{align}
\end{lemma}

\begin{proof}
	It is clear from basic existence theory that for periodic $f$, the following problem has a unique periodic solution (with $\int_{\T^d} udy = 0$)
	\begin{align*}
		-\nabla \cdot \left( w(\cdot\,;z) \nabla  u \right) = f - \int_{\T^d} fdy.
	\end{align*}
	Now, we can represent the map $f \to u$ via an integral kernel (see~\cite{GruterWidman,FJK}), i.e.
	\begin{align*}
		u(x;z) = \int_{\T^d} G(x,y;z) f(y) dy.
	\end{align*}
	It can be shown using standard methods (see~\cite{GruterWidman,FJK}) that such $G$ is unique, periodic, and satisfies \cref{e.green.eq}.

	For the rest of the proof, denote $G_{\textup{per}}$ the periodic Green function and let $G_{\Sigma}$ be the Green function for the unit ball $B_1 = \Sigma$. Let $y \in B_{1/2}$ and denote $u_z = G_{\textup{per}}(\cdot,y;z) - G_{\Sigma}(\cdot,y;z)$, then
	\begin{align*}
		\begin{cases}
			-\nabla \cdot \left( w(\cdot\,;z) \nabla  u_z \right)  =  -1, & \textnormal{in $\Sigma$,}\\
			u_z = G_{\textup{per}}(\cdot,y;z), & \text{on $\partial \Sigma$}.
		\end{cases}
	\end{align*}
	By the maximum principle (see~\cite[Theorem 2.2.3]{FKS}) we know that
	\begin{align*}
		\abs{u_z} \leq C + \sup_{\partial \Sigma} G_{\textup{per}}(\cdot,y;z).
	\end{align*}
	Next, by compactness we know that the right-hand side is bounded by a constant $C(d,\beta)$, i.e.~$\abs{u_z} \leq C$ in $\Sigma$. Hence, it suffices to bound $G_{\Sigma}$ from above.

	By~\cite[Theorem 3.3]{FJK} we have that
	\begin{align*} 
		G_{\Sigma}(x,y\,; z)
		\leq
		C \int_{|x{-}y|}^{1} \frac{r}{w( B_r(x) \, ; z)} \, dr
	\end{align*}
	\textcolor{black}{By symmetry, we may assume that $|y-x|\leq |z-x|$}. First, we have the lower bound
	\begin{align*} 
		w( B_r(x) \, ; z) \geq \frac1C \int_{B_r(x)} |y-z|^{\beta^2} \, dy \geq C (r \vee |x-z|)^{\beta^2} r^{d}
	\end{align*}
	from which we deduce that
	\begin{align*} 
		\lefteqn{
		\int_{|x{-}y|}^{1} \frac{r}{w( B_r(x) \, ; z)} \, dr
		} \qquad &
		\\ \notag
		         &
		\leq
		C \int_{|x{-}y|}^{1} r^{2-d} (r \vee |x{-}z|)^{-\beta^2} \, \frac{dr}{r}
		\\ \notag
		         & \leq
		C |x{-}z|^{-\beta^2} \int_{|x{-}y|}^{|x{-}y| \vee |x{-}z|} r^{2-d}  \, \frac{dr}{r}  + C \int_{|x{-}y| \vee |x{-}z|}^{1} r^{2-d-\beta^2} \, \frac{dr}{r}
		\,.
	\end{align*}
	The  first integral admits the estimate
	\begin{align*} 
		\int_{|x{-}y|}^{|x{-}y| \vee |x{-}z|} r^{2-d}  \, \frac{dr}{r}
		\leq
		C
		\begin{cases}
			|x{-}y|^{2-d} \chi_{\{|x-z| > |x-y| \}}                 & \mbox{if } d > 2 ,   \\
			\log\left(\frac{ |x{-}y| \vee |x{-}z|}{|x{-}y|} \right) & \mbox{if } d =  2	\,
		\end{cases}
	\end{align*}
	and the second (since $d \geq 2$ and $\beta^2 > 0$)
	\begin{align*} 
		\int_{|x{-}y| \vee |x{-}z|}^{1} r^{2-d-\beta^2} \, \frac{dr}{r}
		\leq
		C
		(|x{-}y| \vee |x{-}z|)^{2-d-\beta^2}.
	\end{align*}
	\textcolor{black}{By $|y-z|\leq |x-z|$ we see that $|x-y| \leq 2 |x-z|$, and hence for $d>2$ we have}
	\begin{align*}
		\int_{|x{-}y|}^{1} \frac{r}{w( B_r(x) \, ; z)} \, dr \lesssim & |x-z|^{-\beta^2} |x-y|^{2-d} + |x{-}z|^{2-d-\beta^2} \\
		\lesssim                                                      & |x-z|^{-\beta^2} |x-y|^{2-d}.
	\end{align*}
	In the case $d = 2$ we get (since $|x-z| < C$)
	\begin{align*}
		\int_{|x{-}y|}^{1} \frac{r}{w( B_r(x) \, ; z)} \, dr \lesssim & |x-z|^{-\beta^2} \log \frac{|x-z|}{|x-y|} + |x{-}z|^{-\beta^2} \\
		\lesssim                                                      & |x-z|^{-\beta^2} \log \frac{|x-z|}{|x-y|}
		\lesssim
		|x-z|^{-\beta^2} \log \frac{C}{|x-y|}.
	\end{align*}
	This proves \cref{e.green.bdd} in the case $i=0$, as $x,y,z \in \T^d$ and we know that the right-hand side is bounded from below. Thus, we can dispense with the additive constant in passing from $G_{\Sigma}$ to $G_{\textup{per}}$ by simply choosing $C$ large enough.

	To prove the gradient estimate we first consider the fact that if we are away from the diagonal, i.e.~$x \neq y$ then $v(\cdot) =  G_{\Sigma}(\cdot,y;z)$ solves
	\begin{align*}
		\nabla \cdot (w_z \nabla v) = 0.
	\end{align*}
	As such \cref{l.grad.nopole} implies that whenever $|x-y| \wedge |y-z| > 16 \varrho$ we get
	\begin{align*}
		|\nabla_y G_{\Sigma}(x,y;z)| \leq \frac{1}{{\varrho}} \left (\frac{1}{w(B_{\varrho}(y);z)}\int_{B_{\varrho}(y)} G_{\Sigma}(x,y';z)^2 d\textcolor{black}{\nu}_z(y') \right )^{1/2}.
	\end{align*}
	Assembling the estimates, as $\varrho < 1$, we obtain
	\begin{align*}
		|\nabla_y G_{\Sigma}(x,y;z)| \lesssim \frac{1}{\varrho} G_\Sigma(x,y;z),
	\end{align*}
	and we complete the proof by arguing as for $G_{\textup{per}}$ above and then applying \cref{e.green.bdd} for $i=0$.
\end{proof}

\subsection{Proof of \cref{t.torus.reg}}

Fix~$x^\ast,z^\ast \in \T^d$ such that $x^\ast\neq z^\ast$, and let $r = \min\left(|x^\ast-z^\ast|, 1 \right)/16$\footnote{\textcolor{black}{Note that as we work on $\T^d$, the distance $|x^\ast-z^\ast|$ is always bounded and hence the minimum with 1 is not required. For the reader's convenience, we have taken the minimum with 1 to emphasize that we are interested in the small scales proportional to $|x^\ast-z^\ast|$.}}. \textcolor{black}{In the sequel, we use $w_{z^\ast}$ and $u_{z^\ast}$ to denote functions $w(\cdot;z^\ast)$ and $u(\cdot;z^\ast)$ whenever confusion cannot arise.}

Formally, one could differentiate $u$ with respect to $z$ leading to a new equation and then try to derive $L^\infty$ bounds for the solution. However, then the right-hand side would contain $\partial_{z_i} w$ which is not in $L^2(w)$. To overcome this, we first shift the pole, and then derive a corrector expansion and $L^\infty$ bounds for the correctors. The proof is divided into four steps. In the step 1, we shift the pole, while the corrector expansion is derived in step 2. Correctors in the expansion are then estimated in step 3, and in step 4 we conclude the proof.
\smallskip

\noindent \emph{Step 1. Shifting the pole:}
We introduce the following shift transformation (i.e.~change of coordinates):
Let $\Phi_{z^\ast}(x) = x + \eta(x{-}z^\ast) h$, where $ \chi_{B_{r}} \leq \eta \leq \chi_{B_{2r}}$ with $\| \nabla^k \eta \|_{L^\infty(\R^d)} \leq C_k r^{-k}$. Note that if \textcolor{black}{$|h|\leq cr$} for small constant $c$, then $\Phi_{z^\ast}$ is actually a diffeomorphism.

Next we observe that since $w(x;z) = w(x-z)$, we also have $w(x+h;z+h) = w(x;z)$, which will be useful later.
Ultimately, we will establish an expansion of the form
\begin{align*}
	u(\Phi_{z^\ast}(\cdot);z^\ast) - u(\Phi_{z^\ast}(\cdot);z^\ast+h) = \sum_{1 \leq |\alpha| \leq d/2} v^{(\alpha)}(x^\ast;z^\ast) h^\alpha + o(|h|^{d/2}),
\end{align*}
for $x \in (B_{4r}(z^\ast))^c$ and $v^{(\alpha)}(x^\ast;z^\ast) = \partial_z^\alpha u(x^\ast;z^\ast)$. Then we establish $L^\infty$ bounds for the correctors $v^{(\alpha)}$ and this will give us the desired regularity for $u$ in terms of $z$.

By the definition of $u(x;z)$, we have for any $\phi \in C^\infty(\T^d)$ that
\begin{align*}
	\int_{\T^d} w(x;z^\ast+h) \mathbf f(x;z^\ast + h) \nabla \phi(x) dx  = \int_{\T^d} w(x;z^\ast+h) \nabla u(x;z^\ast+h) \cdot \nabla \phi(x) dx.
\end{align*}
To proceed, denote first $\widehat{w}  := w( \Phi_{z^\ast}(\cdot)\,; z^\ast+h) $, $ \widehat{u} := u( \Phi_{z^\ast}(\cdot)\,; z^\ast+h)$ and $\widehat{\mathbf{f}} := f(\Phi_{z^\ast}(\cdot);z^\ast+h)$.
Considering the change of variables (for $h$ small enough) $x = \Phi_{z^\ast}(y)$, then
$dx = |\det\nabla \Phi_{z^\ast}(y)| dy$ and if we denote $\widehat{\phi}(y) = \phi(\Phi_{z^\ast}(y))$, then substituting $u(x;z^\ast+h) = \widehat{u}(\Phi_{z^\ast}^{-1}(x))$ we have that
\begin{align*}
	\int_{\T^d} w(\Phi_{z^\ast}(y);z^\ast+h) \big [(\nabla \Phi_{z^\ast})^{-T}\nabla_y \widehat{u}(y;z^\ast+h)\big ] \cdot \big [(\nabla \Phi_{z^\ast})^{-T} \nabla_y \widehat{\phi}(y)\big ] |\det(\nabla \Phi_{z^\ast}(\cdot))| dy
	\\
	=
	\int_{\T^d} \widehat{w} (\nabla \Phi_{z^\ast})^{-1} (\nabla \Phi_{z^\ast})^{-T} |\det(\nabla \Phi_{z^\ast}(\cdot))| \nabla \widehat{u} \cdot \nabla \widehat{\phi} dy.
\end{align*}
Thus, we see that $\widehat{u}$ solves the equation
\begin{align*}
	-\nabla \cdot (\widehat{w} (\nabla \Phi_{z^\ast})^{-1} (\nabla \Phi_{z^\ast})^{-T} |\det(\nabla \Phi_{z^\ast}(\cdot))| \nabla \widehat{u}) = \nabla \cdot (|\det(\nabla \Phi_{z^\ast}(\cdot))| \widehat{w} (\nabla\Phi_{z^\ast})^{-1} \widehat{\mathbf{f}}),
\end{align*}
and we get
\begin{align*} 
	- \nabla \cdot \widehat{w} \nabla  \widehat{u}
	=
	\nabla \cdot \a  \nabla \widehat{u}
	+\nabla \cdot (|\det(\nabla \Phi_{z^\ast}(\cdot))| \widehat{w} (\nabla\Phi_{z^\ast})^{-1} \widehat{\mathbf{f}}),
\end{align*}
where
\begin{align*}
	\a  = \widehat{w}(\cdot) \Bigl((\nabla \Phi_{z^\ast}(\cdot))^{-1} (\nabla \Phi_{z^\ast}(\cdot))^{-T} |\det(\nabla \Phi_{z^\ast}(\cdot))|  -\mathbf{I}_d \Bigr)
	\,.
\end{align*}
Note that $\a = 0$ in $B_{r}(z^\ast)$.
Now this equation together with the equation for $u$, 
\begin{align} \label{e.hatuminusu}
	\notag 
	- \nabla \cdot \big( w_{z^\ast} \nabla( \widehat{u} - u_{z^\ast}) \big)
	       & =-\nabla \cdot (w_{z^\ast} \mathbf f_{z^\ast})
	- \nabla \cdot \big( w_{z^\ast} \nabla \widehat{u} \big)
	\\
	\notag &
	=
	\nabla \cdot \big( ( \widehat{w} - w_{z^\ast})  \nabla \widehat{u} \big) +  \nabla \cdot \a  \nabla \widehat{u} + \nabla \cdot (w_{z^\ast}\widetilde{\mathbf{f}})
	\\ &
	=
	\nabla \cdot \big(  w_{z^\ast} \mathbf{b} \nabla( \widehat{u} - u_{z^\ast}) \big)
	+ \nabla \cdot \big( w_{z^\ast} \mathbf{b}  \nabla u_{z^\ast} \big)
	+ \nabla \cdot (w_{z^\ast}\widetilde{\mathbf{f}})
	\,,
\end{align}
where the matrix $\mathbf{b} $ is defined as
\begin{align*} 
	\mathbf{b}
	=
	\left( \frac{\widehat{w}}{w_{z^\ast}} - 1\right) \mathbf{I}_d
	+
	\frac{\widehat{w}}{w_{z^\ast}} \Bigl((\nabla \Phi_{z^\ast}(\cdot))^{-1} (\nabla \Phi_{z^\ast}(\cdot))^{-T} |\det(\nabla \Phi_{z^\ast}(\cdot))|  -\mathbf{I}_d \Bigr)  \,,
\end{align*}
and
\begin{align*}
	w_{z^\ast} \widetilde{\mathbf{f}} = |\det(\nabla \Phi_{z^\ast}(\cdot))| \widehat{w} (\nabla\Phi_{z^\ast})^{-1} \widehat{\mathbf f} - w_{z^\ast} \mathbf f_{z^\ast}.
\end{align*}
Here $\mathbf{b} $ is as smooth as the correlation function. Indeed,
\begin{align*} 
	\frac{\widehat{w}(x)}{w(x\,; z^\ast)}
	=
	\frac{w( x - z^\ast - (1-\eta) h)}{w(x{-}z^\ast)}\,,
\end{align*}
so that in $B_r(z^\ast)$ we have that $\widehat{w}/w_{z^\ast} = 1 $, and hence $\mathbf{b} = 0$ in $B_r(z^\ast)$.
Next observe that by our assumptions on $w$ we have
\begin{align*}
	\frac{\widehat{w}}{w(x;z^\ast)}
	\lesssim
	\left(\frac{r+|h|}{r}\right)^{\beta^2}
\end{align*}
which immediately gives the bounds
\begin{align*}
	\left |\frac{\widehat{w}}{w(x;z^\ast)} - 1 \right | \lesssim \left(1+|h|/r\right)^{\beta^2}-1 = \beta^2 (|h|/r) + o(|h|/r)
\end{align*}
and
\begin{align*}
	\left | \Bigl((\nabla \Phi_{z^\ast}(\cdot))^{-1} (\nabla \Phi_{z^\ast}(\cdot))^{-T} |\det(\nabla \Phi_{z^\ast}(\cdot))| - \mathbf{I}_d \Bigr) \right | \leq c |h|/r.
\end{align*}
Hence, we have
\begin{align}\label{e.bbdd}
	|\mathbf{b}| \leq c |h|/r.
\end{align}
Thus, $\mathbf{b}/|h|$ is bounded, and we can choose $|h|<r/c$ such that $|\mathbf{b}| < 1/2$.

\vspace{1cm}

\noindent \emph{Step 2. Corrector expansion:} In this step we derive the expansions.

First, note that for any function $\tilde u$, denoting $v = \widehat{u} - u_{z^\ast} - \tilde u$ and using \cref{e.hatuminusu}, we obtain
\begin{multline*}
	- \nabla \cdot \big( w_{z^\ast} \nabla v \big) = \nabla \cdot \big(  w_{z^\ast} \mathbf{b} \nabla v \big)
	+ \nabla \cdot \big( w_{z^\ast} \mathbf{b}  \nabla u_{z^\ast} \big) 
	+ \nabla \cdot (w_{z^\ast}(1+\mathbf{b}) \nabla \tilde u )
	\\
	+ \nabla \cdot(w_{z^\ast}\widetilde{\mathbf{f}}).
\end{multline*} 
After rearrangement, this gives
\begin{align}\label{eq.bla}
	- \nabla \cdot  ( w_{z^\ast}(1+\mathbf{b}) \nabla v) = \nabla \cdot \big( w_{z^\ast} \mathbf{b}  \nabla u_{z^\ast} \big) + \nabla \cdot (w_{z^\ast}(1+\mathbf{b}) \nabla \tilde u )+\nabla \cdot(w_{z^\ast}\widetilde{\mathbf{f}}).
\end{align}
Now let us define the expansions (in the variable $h$, and for any $m\in \N$) for the involved terms:
\begin{align}\label{e.expansion}
	\begin{split}
		v_m = & \, \widehat{u} - u_{z^\ast} - \sum_{1 \leq |\alpha| \leq m} v^{(\alpha)} h^\alpha, \quad \mathbf{b} = \sum_{1\leq |\alpha| \leq m} \mathcal{E}^{(\alpha)} h^\alpha + \chi_{(B_r(z^\ast))^c} o(|h|^m) \\
		\widetilde{\mathbf{f}} = & \sum_{1 \leq |\alpha| \leq m} \mathcal{F}^{(\alpha)} h^\alpha + E_{\mathbf f,m},
	\end{split}
\end{align}
where small-o notation $o(|h|^m)$ means that the remainder term may depend on $r$. The error term $E_{\mathbf f,m}$ satisfies for $q$ given in the statement of \cref{t.torus.reg},
\begin{align*}
	\|E_{\mathbf{f},m}\|_{L^q(\T^d;w_{z^\ast})} = o(|h|^m),
\end{align*}
which again may depend on $r$.
Furthermore, using the calculations below \cref{e.hatuminusu} and some elementary computations we obtain bounds for the coefficients in the expansions, namely 
\begin{align} \label{e.coeff.bdd}
	\begin{cases}
		|\mathcal{E}^{(\alpha)}|
		 & \lesssim \frac{1}{|x-z^\ast|^{|\alpha|}} \chi_{B_r^c(z^\ast)} \\
		\|\mathcal{F}^{(\alpha)}\|_{L^q(\T^d;w_{z^\ast})}
		 & \lesssim r^{-|\alpha|} {\displaystyle \sum_{0 \leq |\alpha_1 + \alpha_2| \leq |\alpha|}} \norm{(\partial^{\alpha_1}_x \partial^{\alpha_2}_z \mathbf{f})(\cdot;z^\ast)}_{L^q(\T^d;w_{z^\ast})},
	\end{cases}
\end{align}
and the implied constants are independent of $h$ and $r$.

As the expansion of $v_m$ above is formal, in the following we will define the functions $v^{(\alpha)}$ so that we will have $v_m = o(|h|^m)$ in $L^\infty$. 
Here we have, for the sake of simplicity in the notation, omitted the division by $\alpha!$ (Taylor expansions can be recovered simply by multiplying with $\alpha!$).
Note also that by \cref{e.bbdd}, we do not have a zero order term in the expansion of $\mathbf{b}$ or $\widetilde{\mathbf{f}}$.

Now, by using the definition of $v_m$, we get from \cref{eq.bla} that
\begin{multline}\label{e.master}
	- \nabla \cdot  ( w_{z^\ast}(1+\mathbf{b}) \nabla v_m) = \nabla \cdot \big( w_{z^\ast} \mathbf{b}  \nabla u_{z^\ast} \big) + \sum_{|\alpha| \leq m} \nabla \cdot \left (w_{z^\ast}(1+\mathbf{b}) \nabla v^{(\alpha)}   \right )h^\alpha \\
	+ \nabla \cdot (w_{z^\ast}\widetilde{\mathbf{f}}).
\end{multline}
Now we will define the functions $v^{(\alpha)}$ so that we have $v_m = o(|h|^m)$. To do this, we first expand the right-hand side of \cref{e.master} using the expansion in \cref{e.expansion}, leading to
\begin{multline}\label{e.bladibla}
	- \nabla \cdot  ( w_{z^\ast}(1+\mathbf{b}) \nabla v_m)
	=
	\nabla \cdot \big( w_{z^\ast} \sum_{1 \leq |\alpha|\leq m} \mathcal{E}^{(\alpha)} h^\alpha  \nabla u_{z^\ast} \big)
	\\
	+
	\sum_{|\alpha| \leq m} \nabla \cdot \left (\sum_{1 \leq |\widehat{\alpha}|\leq m } w_{z^\ast}(\mathcal{E}^{(\widehat{\alpha})} h^{\widehat{\alpha}}) \nabla v^{(\alpha)}   \right )h^\alpha
	+
	\sum_{|\alpha| \leq m} \nabla \cdot \left (w_{z^\ast} \nabla v^{(\alpha)}   \right )h^\alpha
	\\
	+
	\sum_{1 \leq |\alpha|\leq m} \nabla \cdot( w_{z^\ast} \mathcal{F}^{(\alpha)}) h^\alpha+o(|h|^{m}).
\end{multline}
Matching the powers of $h^\alpha$ in the right-hand side of the above we get that the functions $v^{(\alpha)}$ should solve 
\begin{multline}\label{e.valphaeq}
	-\nabla \cdot \left (w_{z^\ast} \nabla v^{(\alpha)} \right )
	=
	\nabla \cdot \big( w_{z^\ast} \mathcal{E}^{(\alpha)} \nabla u_{z^\ast} \big)
	+
	\sum_{
	\substack{1 \leq |\alpha_2|, |\alpha_1| \leq m, \\
			\alpha_1+\alpha_2 = \alpha}
	}
	\nabla \cdot \left ( w_{z^\ast} \mathcal{E}^{(\alpha_2)} \nabla v^{(\alpha_1)} \right )
	\\
	+
	\nabla \cdot( w_{z^\ast} \mathcal{F}^{(\alpha)})
\end{multline}
for $1 \leq |\alpha| \leq m$. 
We can now define the functions $v^{(\alpha)}$ inductively by solving \cref{e.valphaeq} in $\T^d$ (with periodic boundary conditions) for $1 \leq |\alpha| \leq m$, where we note that at each step the right-hand side has the correct regularity to ensure solvability. This follows from \cref{e.coeff.bdd}, i.e.~the expansion coefficients $\mathcal{E}^{(\alpha)}$ vanish in $B_r(z^\ast)$, and $\mathcal{F}^{(\alpha)} \in L^q(\T^d;w_{z^\ast})$. Now, inserting our definition of $v^{(\alpha)}$ into \cref{e.bladibla} we get
\begin{align*}
	- \nabla \cdot  ( w_{z^\ast}(1+\mathbf{b}) \nabla v_m) 
	= &
	\sum_{m+1 \leq |\alpha|} h^\alpha \nabla \cdot \big( w_{z^\ast}  \mathcal{E}^{(\alpha)}\nabla u_{z^\ast} \big)
	\\
	&+ \sum_{\substack{1\leq |\alpha| \leq m, 1 \leq |\widehat{\alpha}|\leq m, \\ m+1 \leq |\widehat{\alpha}
	+ \alpha|}} \nabla \cdot \left ( w_{z^\ast} \mathcal{E}^{(\widehat{\alpha})} \nabla v^{(\alpha)}   \right )h^{\widehat{\alpha} + \alpha}
	\\
	&+ \sum_{m+1 \leq |\alpha|} h^\alpha \nabla \cdot( w_{z^\ast} \mathcal{F}^{(\alpha)}).
\end{align*}
In order to show $v_m = o(|h|^m)$, it suffices to use \cref{l.holder} and note that the quantities $\mathcal{E}^{(\alpha)}\nabla u_{z^\ast}$, $\mathcal{E}^{(\widehat{\alpha})} \nabla v^{(\alpha)}$, and $\mathcal{F}^{(\alpha)}$ appearing on the right-hand side belong to $L^{2d-\eps}$, which follows from construction and the estimate \eqref{e.coeff.bdd} together with the assumption of the theorem.

\noindent \emph{Step 3. Local boundedness of the correctors:}
In this step we prove an $L^\infty$ estimate for $v^{(\alpha)}$. Recall that $v^{(\alpha)}$ for $1 \leq |\alpha| \leq m$ solves \cref{e.valphaeq}.
Here, as the right-hand side of \cref{e.valphaeq} is the divergence of a periodic function, we may represent $v^{(\alpha)}$ using the Green function. After integration by parts, we obtain
\begin{align*}
	v^{(\alpha)}(x^\ast) 
	=&
	-\int_{\T^d} \nabla_y G(x^\ast,y;z^\ast) \cdot \big( w_{z^\ast} \mathcal{E}^{(\alpha)} \nabla u_{z^\ast} \big) dy
	\\
	& -
	\sum_{
	\substack{1 \leq |\alpha_2|, |\alpha_1| \leq m, \\
			\alpha_1+\alpha_2 = \alpha}
	}
	\int_{\T^d} \nabla_y G(x^\ast,y;z^\ast) \cdot \left ( w_{z^\ast} \mathcal{E}^{(\alpha_2)} \nabla v^{(\alpha_1)} \right )dy
	\\
	& +
	\int_{\T^d} \nabla_y G(x^\ast,y;z^\ast) w_{z^\ast}(y) \mathcal{F}^{(\alpha)}(y)dy
	\\
	=& I_1 + I_2 + I_3.
\end{align*}
We estimate the terms one by one. For this, let us consider the rings\\ $A_j = (B_{2^{j} r}(x^\ast) \setminus B_{2^{j-1} r}(x^\ast))\setminus B_r(z^\ast)$ for $j=0,\ldots,j_0$, where $A_0 = B_{r}(x^\ast)\setminus B_r(z^\ast) = B_r(x^\ast)$, and $2^{j_0} r = C$\textcolor{black}{, and note that $\mathcal{E}^{(\alpha)}=0$ on the ball $B_r(z^\ast)$.} 
Using \cref{l.green,e.coeff.bdd} we have for $d > 2$ that
\begin{align} \label{e.corrector.I1.1}
	\notag I_1 & = \sum_{j=0}^{j_0} \int_{A_j} \nabla_y G(x^\ast,y;z^\ast) \cdot \big( w_{z^\ast} \mathcal{E}^{(\alpha)} \nabla u_{z^\ast} \big) dy
	\\
	\notag
	& \lesssim
	\sum_{j=0}^{j_0} \int_{A_j} \frac{(|x^\ast-z^\ast| \vee |y - z^\ast|)^{-\beta^2} |x^\ast-y|^{2-d} |y-z^\ast|^{\beta^2} |\mathcal{E}^{(\alpha)} \nabla u_{z^\ast}|}{(|x^\ast-y| \wedge |y-z^\ast|)} dy 
	\\
	& \lesssim
	r^{-|\alpha|} \int_{A_0} |x^\ast-y|^{1-d} |\nabla u_{z^\ast}| dy
	+
	\textcolor{black}{\sum_{j=1} (2^j r)^{1-|\alpha|} \fint_{A_j} |\nabla u_{z^\ast}| d\textcolor{black}{\nu}_{z^\ast},}
\end{align}
where we have also taken advantage of the fact that on the sets $A_j, j > 0$ we are at a distance \textcolor{black}{of the order} $2^j r$ away from the pole \textcolor{black}{$z^\ast$, and hence $\frac{\sup w_{z^\ast}}{\inf w_{z^\ast}} \lesssim C$}.

Let us begin by bounding the first term on the right, and for simplicity let us assume that $x^\ast = 0$. Alongside with \cref{l.holder}, we obtain
\begin{align} \label{e.corrector.I1.2}
	\notag \int_{A_0} |y|^{1-d} |\nabla u_{z^\ast}| dy & = \sum_{k} \int_{B_{2^{-k} r} \setminus B_{2^{-{k+1}} r}} |y|^{1-d} |\nabla u_{z^\ast}| dy
	\\
	\notag & \lesssim
	\sum_{k}
	\frac{|2^{-k} r|}{\abs{B_{2^{-k} r} \setminus B_{2^{-{k+1}} r}}}\int_{B_{2^{-k} r} \setminus B_{2^{-{k+1}} r}} |\nabla u_{z^\ast}| dy
	\\
	\notag & \lesssim
	\sum_{k} |2^{-k} r| \left ( \fint_{B_{2^{-k} r} \setminus B_{2^{-{k+1}} r}} |\nabla u_{z^\ast}|^2 d\textcolor{black}{\nu}_{z^\ast} \right )^{1/2}
	\\
	\notag & \lesssim \sum_{k} |2^{-k} r| (2^{-k} r)^{\gamma-1} \|\mathbf{f}(\cdot;z^\ast)\|_{L^q(\T^d;w_{z^\ast})} \\
	& \lesssim
	r^{\gamma} \|\mathbf{f}(\cdot;z^\ast)\|_{L^q(\T^d;w_{z^\ast})}.
\end{align}
\textcolor{black}{Above, we use the fact that $w_{z^\ast}$ is essentially constant on the annulus $B_{2^{-k} r} \setminus B_{2^{-{k+1}}r}$.}
Repeating a similar argument for the second term in \cref{e.corrector.I1.1} we get
\begin{align} \label{e.corrector.I1.3}
	\sum_{j=1} (2^j r)^{1-|\alpha|} \fint_{A_j} |\nabla u_{z^\ast}| w_{z^\ast} dy \lesssim r^{\gamma-\abs{\alpha}} \|\mathbf{f}(\cdot;z^\ast)\|_{L^q(\T^d;w_{z^\ast})}.
\end{align}
Assembling \cref{e.corrector.I1.1,e.corrector.I1.2,e.corrector.I1.3} finally gives
\begin{align}\label{e.corrector.I1.final}
	I_1 \lesssim r^{\gamma-|\alpha|} \|\mathbf{f}(\cdot;z^\ast)\|_{L^q(\T^d;w_{z^\ast})}.
\end{align}

Repeating a similar argument for $I_2$ we are led to estimate, for $j=j_0,j_0-1,\ldots$, terms
\begin{align*}
	\fint_{A_j} |\nabla v^{(\alpha)}| d\textcolor{black}{\nu}_{z^\ast}.
\end{align*}
Since $v^{(\alpha)}$ solves \cref{e.valphaeq} we get by testing with $v^{(\alpha)}$, Hölder's inequality, \cref{e.coeff.bdd}, and an induction argument that
\begin{align}\label{e.valpha.global}
	\int_{\T^d} |\nabla v^{(\alpha)}|^2 d\textcolor{black}{\nu}_{z^\ast}
	 & \lesssim
	\int_{\T^d} |\mathcal{E}^{(\alpha)}|^2|\nabla u_{z^\ast}|^2 d\textcolor{black}{\nu}_{z^\ast}
	 +
	\sum_{
	\substack{1 \leq |\alpha_2|, |\alpha_1| \leq m, \\
			\alpha_1+\alpha_2 = \alpha}
	}
	\int_{\T^d} |\mathcal{E}^{(\alpha_2)}|^2 |\nabla v^{(\alpha_1)}|^2 d\textcolor{black}{\nu}_{z^\ast}
	\nonumber                                       \\
	 & +
	\int_{\T^d} |\mathcal{F}^{(\alpha)}|^2 d\textcolor{black}{\nu}_{z^\ast}
	\nonumber                                       \\
	 & \lesssim
	r^{-2|\alpha|} 
	{\displaystyle \sum_{0 \leq |\alpha_1 + \alpha_2| \leq |\alpha|}} \norm{(\partial^{\alpha_1}_x \partial^{\alpha_2}_z \mathbf{f})(\cdot;z^\ast)}_{L^q(\T^d;w_{z^\ast})}^2
	.
\end{align}
Applying \cref{l.holder,l.holder2,e.valpha.global,e.coeff.bdd}, and again using an induction argument we get
\begin{align}\label{e.valpha.grad.bdd}
	\fint_{B_\varrho} |\nabla v^{(\alpha)}|^2 d\textcolor{black}{\nu}_{z^\ast}
	 & \lesssim
	\varrho^{2\gamma-2}
	r^{-2|\alpha|}{\displaystyle \sum_{0 \leq |\alpha_1 + \alpha_2| \leq |\alpha|}} \norm{(\partial^{\alpha_1}_x \partial^{\alpha_2}_z \mathbf{f})(\cdot;z^\ast)}_{L^q(\T^d;w_{z^\ast})}^2.
\end{align}
Using \cref{e.valpha.grad.bdd} we can bound $I_2$ similarly as the term $I_1$ and obtain
\begin{align*}
	I_2 \lesssim C(|\alpha|) r^{\gamma-|\alpha|} \left (\|\nabla u\|_{L^2(\T^d,w)} + 
	{\displaystyle \sum_{0 \leq |\alpha_1 + \alpha_2| \leq |\alpha|}} \norm{(\partial^{\alpha_1}_x \partial^{\alpha_2}_z \mathbf{f})(\cdot;z^\ast)}_{L^q(\T^d;w_{z^\ast})}^2 \right ).
\end{align*}
{\color{black} 
Finally, for $I_3$ it suffices to consider the integral over $A = B_{64r}(x^\ast)\setminus B_r(z^\ast)$. Outside $B_{64r}(x^\ast)$, $\Phi_{z^\ast}(x) = x$ and no prefactors of the form $r^{-|\alpha|}$ appears in the integrals, furthermore there is no singularity of the Greens function in this region. We would also like to point out that inside $B_r(z^\ast)$ we have $\Phi_{z^\ast}(x) = x+h$, and $\hat w / w = 1$ (see Step 2) and again no $r^{-|\alpha|}$ prefactor appears.
In order to estimate the integral over $A$, we note that for $q >2d-\varepsilon$, the conjugate $p=\frac{q}{q-1}$ satisfies $p<\frac{d}{d-1}$, provided $\varepsilon<d$, or equivalently, $(1-d)p+d>0$. Hence H\"older's inequality, \cref{e.coeff.bdd} for $\mathcal{F}^{(\alpha)}$, and the pointwise bounds for $\nabla_y G$ in \cref{l.green} as in the case $I_1$ gives 
\begin{align*}
    &\int_{A} \nabla_y G(x^\ast,y;z^\ast) w_{z^\ast}(y) \mathcal{F}^{(\alpha)}(y)dy \\
    &\leq \|\mathcal{F}^{(\alpha)}\|_{L^q(\T^d;w_{z^\ast})}\left[\int_{A}|\nabla_y G(x^\ast,y;z^\ast)|^p w_{z^\ast}(y)dy \right]^{1/p} \\
    &\lesssim r^{-|\alpha|} {\displaystyle \sum_{0 \leq |\alpha_1 + \alpha_2| \leq |\alpha|}} \norm{(\partial^{\alpha_1}_x \partial^{\alpha_2}_z \mathbf{f})(\cdot;z^\ast)}_{L^q(\T^d;w_{z^\ast})}\left[\int_{A}|\nabla_y G(x^\ast,y;z^\ast)|^p w_{z^\ast}(y)dy \right]^{1/p},
\end{align*}
where
\begin{align*}
    \left[\int_{A}|\nabla_y G(x^\ast,y;z^\ast)|^p w_{z^\ast}(y)dy \right]^{1/p} \lesssim \left[\int_{B_r(x^\ast)} |y-x^\ast|^{(1-d)p}dy\right]^{1/p}
    \lesssim r^{1-d+d/p}
\end{align*}
that leads to 
\begin{align*}
I_3 \lesssim r^{\gamma-|\alpha|}{\displaystyle \sum_{0 \leq |\alpha_1 + \alpha_2| \leq |\alpha|}} \norm{(\partial^{\alpha_1}_x \partial^{\alpha_2}_z \mathbf{f})(\cdot;z^\ast)}_{L^q(\T^d;w_{z^\ast})}.
\end{align*}
}
Putting all this together, we finally arrive at the estimate
\begin{align*}
	|v^{(\alpha)}(x^\ast)| \lesssim C(|\alpha|) r^{\gamma-|\alpha|} 
	{\displaystyle \sum_{0 \leq |\alpha_1 + \alpha_2| \leq |\alpha|}} \norm{(\partial^{\alpha_1}_x \partial^{\alpha_2}_z \mathbf{f})(\cdot;z^\ast)}_{L^q(\T^d;w_{z^\ast})}.
\end{align*}
This gives the required estimate in the case $d>2$. For the case $d=2$, by using the arguments above we get the same estimates as in \cref{e.corrector.I1.1,e.corrector.I1.2,e.corrector.I1.3,e.valpha.grad.bdd} except with an extra $\log(1/r)$ factor. As such, we recover the same type of bound, namely
\begin{align*}
	|v^{(\alpha)}(x^\ast)| \lesssim C(|\alpha|) r^{\widehat{\gamma}-|\alpha|} 
	{\displaystyle \sum_{0 \leq |\alpha_1 + \alpha_2| \leq |\alpha|}} \norm{(\partial^{\alpha_1}_x \partial^{\alpha_2}_z \mathbf{f})(\cdot;z^\ast)}_{L^q(\T^d;w_{z^\ast})},
\end{align*}
where now $\widehat{\gamma} < \gamma$.

\vspace{1cm}

\noindent \emph{Step 4. Concluding the proof: }

Putting together all the steps, we have proved that for $z \in \T^d \setminus B_{4r}(x)$, $r= |x-z|/16$ and $|h| < r/2$, it holds
\begin{align*}
	u(x;z) = u(x;z+h) + \sum_{1 \leq |\alpha| \leq m} v^{(\alpha)}_r(x;z) h^{\alpha} + o(|h|^m).
\end{align*}
This implies
\begin{align}\label{e.final}
	\notag
	|\partial_z^{\alpha} u(x;z)| 
	=& 
	|v^{(\alpha)}_r(x;z)|\alpha! 
	\\
	\lesssim &
	{\displaystyle \sum_{0 \leq |\alpha_1 + \alpha_2| \leq |\alpha|}} \norm{(\partial^{\alpha_1}_x \partial^{\alpha_2}_z \mathbf{f})(\cdot;z)}_{L^q(\T^d;w_{z})} |x-z|^{\gamma-|\alpha|}.
\end{align}
To complete the proof, note that if $d$ is even we get immediately from \cref{e.final} that
\begin{align*}
	\|u(x;\cdot) \|_{H^{d/2}(\T^d)} \lesssim 
	{\displaystyle \sum_{0 \leq |\alpha_1 + \alpha_2| \leq d/2}} \left (\int_{\T^d} \norm{\partial^{\alpha_1}_x \partial^{\alpha_2}_z \mathbf{f}(\cdot; z)}_{\T^d;L^q(\T^d;w_{z})}^{p} dz \right )^{1/p}
\end{align*}
for $p > \frac{d}{\gamma}$, and the implied constant does not depend on $x$.

In the case that $d$ is odd, we can use the Brezis-Mironescu~\cite{brezis2018gagliardo} interpolation theorem to estimate
\begin{equation} \label{e.interpolation}
	\|u(x,\cdot)\|_{H^{\frac{d}{2}}} \leq C \|u(x,\cdot)\|_{W_{\frac{d-1}{2},\frac{2d}{d-1}}}^{\frac{1}{2}} \|u(x,\cdot)\|_{W_{\frac{d+1}{2},\frac{2d}{d+1}}}^{\frac{1}{2}}.
\end{equation}
To bound the terms on the right hand side of the above, we integrate \cref{e.final} w.r.t $z$ to get that
\begin{equation*}
	\norm{u(x,\cdot)}_{W^{\frac{d+1}{2},\frac{2d}{d+1}}(\T^d)}^{\frac{2d}{d+1}} \lesssim \sum_{0 \leq |\alpha_1 + \alpha_2| \leq \frac{d+1}{2}} \int_{\T^d} \norm{\partial^{\alpha_1}_x \partial^{\alpha_2}_z \mathbf{f}(\cdot;z)}_{L^{q}(\T^d;w_{z})}^{\frac{2d}{d+1}} |x-z|^{\left (\gamma-\frac{d+1}{2} \right )\frac{2d}{d+1}} dz.
\end{equation*}
Now using Hölder's inequality we get that for $p_+ > \frac{d+1}{2\gamma}$ it holds
\begin{equation*}
	\norm{u(x,\cdot)}_{W^{\frac{d+1}{2},\frac{2d}{d+1}}(\T^d)}^{p_+ \frac{2d}{d+1}} \lesssim \sum_{0 \leq |\alpha_1 + \alpha_2| \leq \frac{d+1}{2}} \int_{\T^d} \norm{\partial^{\alpha_1}_x \partial^{\alpha_2}_z \mathbf{f}(\cdot;z)}_{L^{q}(\T^d;w_{z})}^{p_+ \frac{2d}{d+1}} dz,
\end{equation*}
where the implied constant is independent of $x$. This bounds the first term on the right hand side of \cref{e.interpolation}.
Similar calculations give that for $p_- > \frac{d-1}{2\gamma}$ it holds
\begin{equation*}
	\norm{u(x,\cdot)}_{W^{\frac{d-1}{2},\frac{2d}{d-1}}(\T^d)}^{p_- \frac{2d}{d-1}} \lesssim \sum_{0 \leq |\alpha_1 + \alpha_2| \leq \frac{d-1}{2}} \int_{\T^d} \norm{\partial^{\alpha_1}_x \partial^{\alpha_2}_z \mathbf{f}(\cdot;z)}_{L^{q}(\T^d;w_{z})}^{p_- \frac{2d}{d-1}} dz,
\end{equation*}
where again the implied constant is independent of $x$. This bounds the second term on the right hand side of \cref{e.interpolation}. 
Finally, if 
\begin{equation*}
	\sum_{0 \leq |\alpha_1 + \alpha_2| \leq \frac{d+1}{2}} \int_{\T^d} \norm{\partial^{\alpha_1}_x \partial^{\alpha_2}_z \mathbf{f}(\cdot;z)}_{L^{q}(\T^d;w_{z})}^{p} dz < \infty,
\end{equation*}
for some $p > \frac{d}{\gamma}$, then the right hand side of \cref{e.interpolation} is finite. This concludes the proof of \cref{t.torus.reg}.

\appendix

\section{Muckenhoupt weights} \label{a.muckenhoupt}

\begin{definition}
	If $0<w<\infty$ a.e., we say that $w\in A_1$ (i.e., belongs to the \emph{Muckenhoupt class $A_1$}) if there exists
	a constant $c(1,w)\ge 1$ such that
	\begin{align*}
		\fint_B w\,dx\le c(1,w)\operatorname{ess\,inf}\displaylimits_B w
	\end{align*}
	for every ball $B\in\R^d$.

	Furthermore, we say that $w\in A_p$, $1<p<\infty$ if there exists a constant $c(p,w)>0$ such that
	\begin{align*}
		\fint_B w\,dx\le c(p,w)\left(\fint_B w^{1/(1-p)}\,dx\right)^{1-p}
	\end{align*}
	for every ball $B\in\R^d$.
\end{definition}

We have that $A_1\subset A_p$. If $w\in A_p$ and $h\in\R^d$, $a\in\R\setminus\{0\}$, then $w(\cdot+h)\in A_p$ and $w(a\cdot)\in A_p$.

\begin{example}\label{ex:Muckenhoupt}
	For $p>1$, We have that $|\cdot|^\delta\in A_p$ if and only if $-d<\delta<d(p-1)$. Also, we have that $|\cdot|^\delta\in A_1$ if and only if $-d<\delta\le 0$. We refer to~\cite[Chapter IX, Proposition 3.2, Corollary 4.4]{Torchinsky1986} for the well-known proofs.
\end{example}

\begin{lemma}\label{lem:Muckenhoupt}
	Let $U\subset\R^d$ be an open bounded domain. Let $\beta\in (0,\sqrt{d})$ and let
	\begin{align*}
		R(x\,;z)
		=
		e^{g(x,z)}|x{-}z|^{-\beta^2 \phi(x) \phi(z)},\quad x,z\in\R^d,
	\end{align*}
	where
	$g\in C_b(\R^d\times\R^d)$, and where $\phi$ is a compactly supported smooth function such that $\chi_{U'} \leq \phi \leq \chi_{U}$ with $U' := \{x \in U \, : \, \dist(x,\partial U)>r_0\}$.

	Then $R(\cdot\,;z)\in A_1$ for any $z\in\R^d$.
\end{lemma}
\begin{proof}
	Expand $R$ as follows,
	\begin{align*}
		R(x\,;z)=\exp(g(x,z)-\beta^2(\phi(x)-\phi(z))\phi(z)\log|x{-}z|-\beta^2\phi^2(z)\log|x{-}z|).
	\end{align*}
	Clearly, by translation invariance and Example~\ref{ex:Muckenhoupt}, $x\mapsto\exp(-\beta^2\phi^2(z)\log|x{-}z|)$ is an element of $A_1$ for any $z$.
	By the Coifman-Rochberg characterization of $A_1$-weights~\cite[Chapter IX, Theorem 3.4]{Torchinsky1986}, it suffices to prove that
	\begin{equation}\label{eq:weightfactor}
		x\mapsto\exp(g(x,z)-\beta^2(\phi(x)-\phi(z))\phi(z)\log|x{-}z|)
	\end{equation}
	is bounded above and below away from zero. By the mean value theorem,
	there exists $\theta=\theta_{x,z}\in (0,1)$, such that
	\begin{align*}
		&|g(x,z)-\beta^2(\phi(x)-\phi(z))\phi(z)\log|x{-}z||\\
		\le&\|g\|_\infty +\beta^2\|\phi\|_\infty |\nabla\phi(\theta x+(1-\theta)z)||x{-}z||\log|x{-}z||\\
		\le&\|g\|_\infty +\begin{cases}\beta^2\operatorname{Lip}(\phi)|x{-}z||\log|x{-}z||,&\quad\text{if}\quad\theta x+(1-\theta)z\in U,\\
		0,&\quad\text{if}\quad\theta x+(1-\theta)z\in\R^d\setminus U.\end{cases}
	\end{align*}
	However, $|x{-}z||\log|x{-}z|| \chi_U(\theta x+(1-\theta)z)$ is bounded as is $U$. Hence the factor~\eqref{eq:weightfactor} is bounded from above and below away from zero.
\end{proof}

\begin{corollary}
	Let $\eta\in C_0^\infty(\R^d)$, $\eta\ge 0$, $\int_{\R^d}\eta(x)\,dx=1$. Then
	the weight
	\begin{align*}
		x\mapsto (R(\cdot\,;z)\ast \eta)(x):=\int_{\R^d}R(x{-}y\,;z)\eta(y)\,dy
	\end{align*}
	belongs to the class $A_1$ for any $z$.
\end{corollary}
\begin{proof}
	By Lemma~\ref{lem:Muckenhoupt}, we may apply~\cite[Lemma 2.1]{Kinnunen1998}.
\end{proof}

\bibliographystyle{abbrv}
\bibliography{geothermal}

\begin{thebibliography}{10}

\bibitem{AKNSTV-1}
B.~Avelin, T.~Kuusi, P.~Nummi, E.~Saksman, J.~M. T\"olle, and L.~Viitasaari.
\newblock {1D stochastic pressure equation with log-correlated Gaussian
  coefficients}.
\newblock {\em Preprint}, pages 1--37, 2024.
\newblock \url{https://arxiv.org/abs/1602.07323}.

\bibitem{berestycki2017}
N.~Berestycki.
\newblock {An elementary approach to Gaussian multiplicative chaos}.
\newblock {\em Electronic Communications in Probability}, 22:1 -- 12, 2017.

\bibitem{brezis2018gagliardo}
H.~Brezis and P.~Mironescu.
\newblock Gagliardo-{N}irenberg inequalities and non-inequalities: the full
  story.
\newblock {\em Annales de l'Institut Henri Poincar{\'e} C, Analyse non
  lin{\'e}aire}, 35(5):1355--1376, 2018.

\bibitem{Bruned_et_al}
Y.~Bruned, A.~Chandra, I.~Chevyrev, and M.~Hairer.
\newblock Renormalising {SPDE}s in regularity structures.
\newblock {\em J. Eur. Math. Soc. (JEMS)}, 23(3):869--947, 2021.

\bibitem{BHZ}
Y.~Bruned, M.~Hairer, and L.~Zambotti.
\newblock Algebraic renormalisation of regularity structures.
\newblock {\em Invent. Math.}, 215(3):1039--1156, 2019.

\bibitem{Duplantier2017}
B.~Duplantier, R.~Rhodes, S.~Sheffield, and V.~Vargas.
\newblock Log-correlated {G}aussian fields: {A}n overview.
\newblock In {\em J.-B. Bost, H. Hofer, F. Labourie, Y. Le Jan, X. Ma, W. Zhang
  (eds), Geometry, Analysis and Probability: In Honor of Jean-Michel Bismut},
  pages 191--216. Springer International Publishing, Cham, 2017.

\bibitem{FJK}
E.~B. Fabes, D.~S. Jerison, and C.~E. Kenig.
\newblock The {W}iener test for degenerate elliptic equations.
\newblock {\em Annales de l'Institut Fourier}, 32(3):151--182, 1982.

\bibitem{FKS}
E.~B. Fabes, C.~E. Kenig, and R.~P. Serapioni.
\newblock The local regularity of solutions of degenerate elliptic equations.
\newblock {\em Communications in Partial Differential Equations}, 7(1):77--116,
  1982.

\bibitem{FW:92}
W.~E. Fitzgibbon and M.~F. Wheeler, editors.
\newblock {\em Modeling and analysis of diffusive and advective processes in
  geosciences}.
\newblock Society for Industrial and Applied Mathematics (SIAM), Philadelphia,
  PA, 1992.

\bibitem{FH:20}
P.~K. Friz and M.~Hairer.
\newblock {\em A course on rough paths}.
\newblock Universitext. Springer, Cham, second edition, 2020.
\newblock With an introduction to regularity structures.

\bibitem{GruterWidman}
M.~Gr{\"u}ter and K.-O. Widman.
\newblock The {G}reen function for uniformly elliptic equations.
\newblock {\em Manuscripta mathematica}, 37(3):303--342, 1982.

\bibitem{Gubinelli_et_al}
M.~Gubinelli, P.~Imkeller, and N.~Perkowski.
\newblock Paracontrolled distributions and singular {PDE}s.
\newblock {\em Forum Math. Pi}, 3:e6, 75, 2015.

\bibitem{Hairer}
M.~Hairer.
\newblock A theory of regularity structures.
\newblock {\em Invent. Math.}, 198(2):269--504, 2014.

\bibitem{HanLin}
Q.~Han and F.~Lin.
\newblock {\em Elliptic partial differential equations}, volume~1.
\newblock American Mathematical Soc., 2011.

\bibitem{HKM}
J.~Heinonen, T.~Kilpel\"{a}inen, and O.~Martio.
\newblock {\em Nonlinear potential theory of degenerate elliptic equations}.
\newblock Dover Publications, Inc., Mineola, NY, 2006.
\newblock Unabridged republication of the 1993 original.

\bibitem{H}
T.~Hewett et~al.
\newblock Fractal distributions of reservoir heterogeneity and their influence
  on fluid transport.
\newblock In {\em SPE Annual Technical Conference and Exhibition}. Society of
  Petroleum Engineers, 1986.

\bibitem{hida2013white}
T.~Hida, H.~Kuo, J.~Potthoff, and L.~Streit.
\newblock {\em White Noise: {A}n Infinite Dimensional Calculus}.
\newblock Mathematics and Its Applications. Springer Netherlands, 2013.

\bibitem{Holden_et_al}
H.~Holden, B.~{\O}{}ksendal, J.~Ub\o{}e, and T.~Zhang.
\newblock {\em Stochastic partial differential equations}.
\newblock Universitext. Springer, New York, second edition, 2010.
\newblock A modeling, white noise functional approach.

\bibitem{HuOe1996}
Y.~Hu and B.~{\O}{}ksendal.
\newblock Wick approximation of quasilinear stochastic differential equations.
\newblock In {\em Stochastic analysis and related topics, {V} ({S}ilivri,
  1994)}, volume~38 of {\em Progr. Probab.}, pages 203--231. Birkh\"{a}user
  Boston, Boston, MA, 1996.

\bibitem{Hu2009Wick}
Y.~Hu and J.-A. Yan.
\newblock Wick calculus for nonlinear gaussian functionals.
\newblock {\em Acta Mathematicae Applicatae Sinica, English Series},
  25:399--414, 2009.

\bibitem{I:18}
L.~Isserlis.
\newblock {On a formula for the product-moment coefficient of any order of a
  normal frequency distribution in any number of variables}.
\newblock {\em Biometrika}, 12(1--2):134--139, 1918.

\bibitem{Janson1997}
S.~Janson.
\newblock {\em Gaussian {H}ilbert spaces}, volume 129 of {\em Cambridge Tracts
  in Mathematics}.
\newblock Cambridge University Press, Cambridge, 1997.

\bibitem{junnila2020}
J.~Junnila, E.~Saksman, and C.~Webb.
\newblock Imaginary multiplicative chaos: Moments, regularity and connections
  to the ising model.
\newblock {\em Ann. Appl. Probab.}, 30(5):2099--2164, 10 2020.

\bibitem{Kinnunen1998}
J.~Kinnunen.
\newblock A stability result on {M}uckenhoupt's weights.
\newblock {\em Publ. Mat.}, 42(1):153--163, 1998.

\bibitem{kondratiev1996}
Y.~G. Kondratiev, P.~Leukert, and L.~Streit.
\newblock Wick calculus in {G}aussian analysis.
\newblock {\em Acta applicandae mathematicae}, 44(3):269--294, 1996.

\bibitem{Kukkonen_2021}
I.~T. Kukkonen and M.~Pentti.
\newblock St1 deep heat project: Geothermal energy to the district heating
  network in espoo.
\newblock {\em IOP Conference Series: Earth and Environmental Science},
  703(1):012035, 2021.

\bibitem{Kupiainen}
A.~Kupiainen.
\newblock Renormalization group and stochastic {PDE}s.
\newblock {\em Ann. Henri Poincar\'{e}}, 17(3):497--535, 2016.

\bibitem{KRV}
A.~Kupiainen, R.~Rhodes, and V.~Vargas.
\newblock Local conformal structure of {L}iouville quantum gravity.
\newblock {\em Comm. Math. Phys.}, 371(3):1005--1069, 2019.

\bibitem{L91}
P.~Leary.
\newblock Deep borehole log evidence for fractal distribution of fractures in
  crystalline rock.
\newblock {\em Geophysical Journal International}, 107(3):615--627, 1991.

\bibitem{leary2020physical}
P.~Leary, P.~Malin, and T.~Saarno.
\newblock A physical basis for the {G}utenberg-{R}ichter fractal scaling.
\newblock In {\em Proceedings of the 45rd Workshop on Geothermal Reservoir
  Engineering, Stanford University, Stanford, CA, USA}, pages 10--12, 2020.

\bibitem{L90}
P.~C. Leary.
\newblock Basement rock fracture structure from {C}ajon {P}ass, {C}alifornia
  and {S}iljan {R}ing, {S}weden borehole geophysical logs.
\newblock In {\em SEG Technical Program Expanded Abstracts 1990}, pages
  153--155. Society of Exploration Geophysicists, 1990.

\bibitem{L:90}
T.~L\'{e}vy.
\newblock Filtration in a porous fissured rock: influence of the fissures
  connexity.
\newblock {\em European J. Mech. B Fluids}, 9(4):309--327, 1990.

\bibitem{LW:94}
T.~L\'{e}vy and J.-C. Wodi\'{e}.
\newblock A model for porous rocks with thin fissures.
\newblock {\em European J. Mech. B Fluids}, 13(3):261--273, 1994.

\bibitem{SheffieldSurvey}
A.~Lodhia, S.~Sheffield, X.~Sun, S.~S. Watson, et~al.
\newblock Fractional {G}aussian fields: a survey.
\newblock {\em Probability Surveys}, 13:1--56, 2016.

\bibitem{Oh2020}
T.~Oh, K.~Seong, and L.~Tolomeo.
\newblock A remark on {G}ibbs measures with log-correlated {G}aussian fields.
\newblock {\em Forum of Mathematics, Sigma}, 12:e50, 2024.

\bibitem{OttoWeber}
F.~Otto and H.~Weber.
\newblock Quasilinear {SPDE}s via rough paths.
\newblock {\em Arch. Ration. Mech. Anal.}, 232(2):873--950, 2019.

\bibitem{Rhodes-Vargas-2014}
R.~Rhodes and V.~Vargas.
\newblock Gaussian multiplicative chaos and applications: a review.
\newblock {\em Probab. Surv.}, 11:315--392, 2014.

\bibitem{Rhodes:2016hin}
R.~Rhodes and V.~Vargas.
\newblock {Lecture notes on {G}aussian multiplicative chaos and {L}iouville
  Quantum Gravity}.
\newblock {\em Preprint}, pages 1--23, 2016.
\newblock \url{https://arxiv.org/abs/1602.07323}.

\bibitem{rhodes2016}
R.~Rhodes and V.~vargas.
\newblock Lecture notes on gaussian multiplicative chaos and liouville quantum
  gravity, 2016.

\bibitem{RunstSickel1996}
T.~Runst and W.~Sickel.
\newblock {\em Sobolev spaces of fractional order, {N}emytskij operators, and
  nonlinear partial differential equations}.
\newblock De Gruyter, Berlin, New York, 1996.

\bibitem{T}
J.~Todoeschuck, O.~Jensen, and S.~Labonte.
\newblock Gaussian scaling noise model of seismic reflection sequences:
  {E}vidence from well logs.
\newblock {\em Geophysics}, 55(4):480--484, 1990.

\bibitem{Torchinsky1986}
A.~Torchinsky.
\newblock {\em {R}eal-variable Methods in Harmonic Analysis}.
\newblock Academic Press, Orlando, FL, 1986.

\bibitem{Wick}
G.~C. Wick.
\newblock The evaluation of the collision matrix.
\newblock {\em Phys. Rev. (2)}, 80:268--272, 1950.

\end{thebibliography}

\end{document}